\documentclass{article}
\usepackage{authblk}
\usepackage{appendix}
\usepackage{tocloft}
\usepackage{a4wide}
\usepackage{amsthm}
\usepackage{graphicx}
\usepackage{amssymb}
\usepackage{amsmath}
\usepackage{ascmac}
\usepackage{setspace}
\usepackage{float}
\usepackage[dvips,usenames]{color}
\usepackage{colortbl}
\usepackage{algorithm}
\usepackage{algorithmic}
\usepackage{setspace}
\newtheorem{Theorem}[equation]{Theorem}
\newtheorem{Corollary}[equation]{Corollary}
\newtheorem{Lemma}[equation]{Lemma}
\newtheorem{Proposition}[equation]{Proposition}

\theoremstyle{definition}
\newtheorem{Definition}[equation]{Definition}

\theoremstyle{remark}
\newtheorem{Remark}[equation]{Remark}
\numberwithin{equation}{section}
\newtheorem{Claim}[equation]{Claim}

\DeclareMathOperator{\End}{End}

\DeclareMathOperator{\ev}{ev}
\DeclareMathOperator{\id}{id}

\DeclareMathOperator{\ad}{ad}

\DeclareMathOperator{\str}{str}

\newcommand{\plim}[1][]{\mathop{\varprojlim}\limits_{#1}}
\newcommand{\ve}{\varepsilon}

\newcommand{\tss}{\hspace{1pt}}

\allowdisplaybreaks

\begin{document}
\title{Affine Super Yangians and Rectangular $W$-superalgebras}
\author{Mamoru Ueda\thanks{udmamoru@kurims.kyoto-u.ac.jp}}
\affil{Research Institute for Mathematical Sciences, Kyoto University, Kyoto 606-8502, Japan}
\date{}
\maketitle
\begin{abstract}
Motivated by the AGT conjecture, we construct a homomorphism from the affine super Yangian $Y_{\ve_1,\ve_2}(\widehat{\mathfrak{sl}}(m|n))$ to the universal enveloping algebra of the rectangular $W$-superalgebra $\mathcal{W}^{k}(\mathfrak{gl}(ml|nl),(l^{(m|n)}))$ for all $m\neq n,m,n\geq2$ or $m\geq3,n=0$. Furthermore, we show that the image of this homomorphism is dense provided that $k+(m-n)(l-1)\neq0$.
\end{abstract}
\section{Introduction}
A $W$-algebra $\mathcal{W}^k(\mathfrak{g},f)$ is a vertex algebra associated with a finite dimensional reductive Lie algebra $\mathfrak{g}$ and a nilpotent element $f\in\mathfrak{g}$. It appeared in the study of two dimensional conformal field theories (\cite{Z}) and has been studied by both physicists and mathematicians since 1980's. When $\mathfrak{g}$ is $\mathfrak{sl}_2$ and $f$ is a nonzero nilpotent element of $\mathfrak{g}$, $\mathcal{W}^k(\mathfrak{g},f)$ is nothing but the Virasoro algebra. However, when $\mathfrak{g}$ is a general finite dimensional reductive Lie algebra, it is no longer a Lie algebra and a presentation by generators and relations is not known in general. On the other hand, this problem has been solved (\cite{BK}, \cite{SKV}) in the case for the finite $W$-algebras (\cite{Pr}). A finite $W$-algebra $\mathcal{W}^{\text{fin}}(\mathfrak{g},f)$ is an associative algebra associated with a finite dimensional reductive Lie algebra $\mathfrak{g}$ and a nilpotent element $f\in\mathfrak{g}$ and it can be regarded as a finite analogue of $W$-algebra $\mathcal{W}^k(\mathfrak{g},f)$ (\cite{DSK1}, \cite{A1}). In \cite{BK}, Brundan and Kleshchev resolved the problem by using a relationship between Yangians of type $A$ and finite $W$-algebras of type $A$.

The Yangian is a quantum group which is a deformation of the current algebra $\mathfrak{g}\otimes\mathbb{C}[z]$. Drinfeld (\cite{D1}, \cite{D2}) introduced a Yangian associated with a finite dimensional simple Lie algebra $\mathfrak{g}$ in order to solve the Yang-Baxter equation. The Yangian of type $A$ has several presentations; the RTT presentation, the parabolic presentation, the Drinfeld presentation and the Drinfeld $J$ presentation. It was shown in \cite{RS} that there exist surjective homomorphisms from Yangians of type $A$ to finite rectangular $W$-algebras of type $A$. The homomorphism is given by the Drinfeld $J$ presentation. More generally, Brundan and Kleshchev (\cite{BK}) constructed a surjective homomorphism from a shifted Yangian, a subalgebra of the Yangian of type $A$, to an arbitrary finite $W$-algebra of type $A$ by using the parabolic presentation. Moreover, the defining relations of finite $W$-algebras of type $A$ have been written down explicitly as a quotient of shifted Yangians in \cite{BK}.

Similar results are known in the super setting. We can define the $W$-superalgebras and finite $W$-superalgebras, which are attached with finite dimensional reductive Lie superalgebras $\mathfrak{g}$ and nilpotent elements of $\mathfrak{g}$ in the even parity. In the case of the Lie superalgebra $\mathfrak{sl}(m|n)$, the corresponding Yangian in the Drinfeld presentation was first introduced by Stukopin (\cite{S}, see also \cite{G}). It is called the super Yangian. A relationship between super Yangians and finite $W$-superalgebras was constructed by Briot and Ragoucy \cite{BR} for the rectangular case and by Peng \cite{Pe} for more general case.

It is natural to ask whether there exists a similar result in the affine setting. The definition of Yangian naturally extends to the case that $\mathfrak{g}$ is a Kac-Moody Lie algebra in the Drinfeld presentation. In the case that $\mathfrak{g}$ is an affine Lie algebra, it is a deformation of the universal enveloping algebra of the current algebra of $\mathfrak{g}$ (see \cite{GNW}, \cite{BL}, and \cite{U1}). Unfortunately, the affine Yangian does not have the Drinfeld $J$ presentation nor the parabolic presentation. Thus, we cannot construct the relationship between $W$-algebras and the affine Yangians in a similar way as \cite{RS} or \cite{BK}. A breakthrough was given by Schiffman and Vasserot (\cite{SV}), who have constructed, using a geometric realization of the Yangian, a surjective homomorphism from the Yangian of $\widehat{\mathfrak{gl}}(1)$ to the universal enveloping algebras of the principal $W$-algebras of type $A$ and have proved the celebrated AGT conjecture (\cite{Ga}, \cite{BFFR}). Gaberdiel, Li, Peng and Zhang (\cite{GLPZ}) defined the Yangian for the affine Lie superalgebra $\widehat{\mathfrak{gl}}(1|1)$ and obtained a result similar to \cite{SV} in the super setting.

In this article, we give a result similar to the one of \cite{RS} in the affine super setting. The corresponding Yangian is the affine super Yangian, which is the deformation of the universal enveloping algebra of the current algebra of $\widehat{\mathfrak{sl}}(m|n)$ (see \cite{U2}).
We construct a homomorphism from the affne super Yangian $Y_{\ve_1,\ve_2}(\widehat{\mathfrak{sl}}(m|n))$ to the universal enveloping algebra (see \cite{FZ} and \cite{MNT}) of $\mathcal{W}^{k}(\mathfrak{gl}(ml|nl),(l^{(m|n)}))$. The following theorem is the main result of this paper.
\begin{Theorem}\label{t1}
Suppose that $m, n\geq2,m\neq n$ or $m\geq3,n=0$, and assume that $l\geq2$ and
\begin{gather*}
\ve_1=\dfrac{\alpha}{m-n},\quad\ve_2=-1-\dfrac{\alpha}{m-n}.
\end{gather*}
Then, there exists an algebra homomorphism 
\begin{equation*}
\Phi\colon Y_{\ve_1,\ve_2}(\widehat{\mathfrak{sl}}(m|n))\to \mathcal{U}(\mathcal{W}^{k}(\mathfrak{gl}(ml|nl),(l^{(m|n)}))),
\end{equation*} 
where $\mathcal{U}(\mathcal{W}^{k}(\mathfrak{gl}(ml|nl),(l^{(m|n)})))$ is the universal enveloping algebra of $\mathcal{W}^{k}(\mathfrak{gl}(ml|nl),(l^{(m|n)}))$. Moreover, the image of $\Phi$ is dense in $\mathcal{U}(\mathcal{W}^{k}(\mathfrak{gl}(ml|nl),(l^{(m|n)})))$ provided that $\alpha\neq0$.
\end{Theorem}
By Theorem~\ref{t1}, provided that $\alpha\neq0$, any irreducible representation of $\mathcal{W}^{k}(\mathfrak{gl}(ml|nl),(l^{(m|n)}))$ can be seen as an irreducible representation of $Y_{\ve_1,\ve_2}(\widehat{\mathfrak{sl}}(m|n))$. In the case that $l=1$, the corresponding result was previously shown in \cite{Gu1}, \cite{K1}, \cite{K0}, \cite{K2}, \cite{U2} and \cite{U3}. 

We expect that the above result will be useful for studying the AGT correspondence for parabolic sheaves. Precisely speaking, Feigin-Finkelberg-Negut-Rybnikov \cite{FFNR} constructed an action of the Guay's affine Yangian on the equivariant cohomology for the affine Laumon spaces. Showing that the kernel of $\Phi$ acts trivially on the equivariant cohomology for the affine Laumon spaces, we obtain the action of the rectangular $W$-algebra on the equivariant cohomology for the affine Laumon spaces. See \cite{N} for the corresponding result in the quantum toroidal setting. 

This paper is organized as follows. In Section 2, we recall the definition of the affine super Yangian and its evaluation map. In Section 3, we recall the definition of the rectangular $W$-superalgebras of type $A$ and construct the elements $W^{(1)}_{i,j}$ and $W^{(2)}_{i,j}$ which are in fact generators of the rectangular $W$-superalgebra of type $A$. In Section 4, we compute OPEs, which are needed for the construction of $\Phi$. In Section 5, we construct an algebra homomorphism from the affine super Yangian to the universal enveloping algebra of the $W$-superalgebras of type $A$. The appendix is devoted to the proof of the fact that $W^{(1)}_{i,j}$ and $W^{(2)}_{i,j}$ generate the rectangular $W$-superalgebra. 

\section{Affine Super Yangians}
First, we recall the definition of the affine super Yangian (see \cite{U2} Definition~3.1). In this paper, we set $\{x,y\}$ as $xy+yx$. We also fix $m,n\in\mathbb{Z}_{\geq0}$ and set the following notation;
\begin{gather*}
 p(i)=\begin{cases}
 0&(1\leq i\leq m),\\
 1&(m+1\leq i\leq m+n).
 \end{cases}
 \end{gather*}
\begin{Definition}\label{Def}
Suppose that $m, n\geq2$ and $m\neq n$. The affine super Yangian $Y_{\ve_1,\ve_2}(\widehat{\mathfrak{sl}}(m|n))$ is the associative superalgebra over $\mathbb{C}$ generated by $x_{i,r}^{+}, x_{i,r}^{-}, h_{i,r}$ $(i \in \{0,1,\cdots, m+n-1\}, r \in \mathbb{Z}_{\geq 0})$ with two parameters $\ve_1, \ve_2 \in \mathbb{C}$ subject to the following defining relations:
\begin{gather}
	[h_{i,r}, h_{j,s}] = 0, \label{eq1.1}\\
	[x_{i,r}^{+}, x_{j,s}^{-}] = \delta_{ij} h_{i, r+s}, \label{eq1.2}\\
	[h_{i,0}, x_{j,r}^{\pm}] = \pm a_{ij} x_{j,r}^{\pm},\label{eq1.3}\\
	[h_{i, r+1}, x_{j, s}^{\pm}] - [h_{i, r}, x_{j, s+1}^{\pm}] 
	= \pm a_{ij} \dfrac{\varepsilon_1 + \varepsilon_2}{2} \{h_{i, r}, x_{j, s}^{\pm}\} 
	- m_{ij} \dfrac{\varepsilon_1 - \varepsilon_2}{2} [h_{i, r}, x_{j, s}^{\pm}],\label{eq1.4}\\
	[x_{i, r+1}^{\pm}, x_{j, s}^{\pm}] - [x_{i, r}^{\pm}, x_{j, s+1}^{\pm}] 
	= \pm a_{ij}\dfrac{\varepsilon_1 + \varepsilon_2}{2} \{x_{i, r}^{\pm}, x_{j, s}^{\pm}\} 
	- m_{ij} \dfrac{\varepsilon_1 - \varepsilon_2}{2} [x_{i, r}^{\pm}, x_{j, s}^{\pm}],\label{eq1.5}\\
	\sum_{w \in \mathfrak{S}_{1 + |a_{ij}|}}[x_{i,r_{w(1)}}^{\pm}, [x_{i,r_{w(2)}}^{\pm}, \dots, [x_{i,r_{w(1 + |a_{ij}|)}}^{\pm}, x_{j,s}^{\pm}]\dots]] = 0\  (i \neq j),\label{eq1.6}\\
	[x^\pm_{i,r},x^\pm_{i,s}]=0\ (i=0, m),\label{eq1.7}\\
	[[x^\pm_{i-1,r},x^\pm_{i,0}],[x^\pm_{i,0},x^\pm_{i+1,s}]]=0\ (i=0, m),\label{eq1.8}
\end{gather}
where\begin{gather*}
a_{ij} =
	\begin{cases}
	{(-1)}^{p(i)}+{(-1)}^{p(i+1)}  &\text{if } i=j, \\
	         -{(-1)}^{p(i+1)}&\text{if }j=i+1,\\
	         -{(-1)}^{p(i)}&\text{if }j=i-1,\\
	        1 &\text{if }(i,j)=(0,m+n-1),(m+n-1,0),\\
		0  &\text{otherwise,}
	\end{cases}\\
	 m_{i,j}=
	\begin{cases}
	-{(-1)}^{p(i+1)} &\text{if } i=j + 1,\\
		{(-1)}^{p(i)} &\text{if } i=j - 1,\\
	        -1 &\text{if }(i,j)=(0,m+n-1),\\
	        1 &\text{if }(i,j)=(m+n-1,0),\\
		0  &\text{otherwise,}
	\end{cases}
\end{gather*}
and the generators $x^\pm_{m, r}$ and $x^\pm_{0, r}$ are odd and all other generators are even and we define $x^\pm_{-1,0}$ as $x^\pm_{m+n-1,0}$.
\end{Definition}

Note that in Definition~\ref{Def} the number of generators of the affine super Yangian is infinite. It is possible to give a presentation of the affine super Yangian such that the number of generators is finite, as we explain below. 

First, we show that $Y_{\ve_1,\ve_2}(\widehat{\mathfrak{sl}}(m|n))$ is generated by $x_{i,r}^{+}, x_{i,r}^{-}, h_{i,r}$ $(i\in\{0,1,\cdots,m+n-1\}, r = 0,1)$. Let $\tilde{h}_{i,1}$ be ${h}_{i,1} - \dfrac{\ve_1 + \ve_2}{2} h_{i,0}^2$. When $r=0$, the relation \eqref{eq1.4} is equivalent to
\begin{equation}
[\tilde{h}_{i,1}, x_{j,s}^{\pm}] = \pm a_{ij}\left(x_{j,s+1}^{\pm}-m_{ij}\dfrac{\varepsilon_1 - \varepsilon_2}{2} x_{j, s}^{\pm}\right).\label{11111}
\end{equation}
By \eqref{11111} and \eqref{eq1.2}, we have the following two relations for all $r\geq1$;
\begin{gather}
x^\pm_{i,r+1}=\pm\dfrac{1}{a_{i,i}}[\tilde{h}_{i,1},x^\pm_{i,r}],\  h_{i,r+1}=[x^+_{i,r+1},x^-_{i,0}]\quad\text{if}\quad i\neq m,0,\label{eq1297}\\
x^\pm_{i,r+1}=\pm\dfrac{1}{a_{i+1,i}}[\tilde{h}_{i+1,1},x^\pm_{i,r}]+m_{i+1,i}\dfrac{\varepsilon_1 - \varepsilon_2}{2} x_{i, r}^{\pm},\   h_{i,r+1}=[x^+_{i,r+1},x^-_{i,0}]\quad\text{if}\quad i=m,0.\label{eq1298}
\end{gather}
Thus, using \eqref{eq1297} and \eqref{eq1298}, $\{h_{i,r},x^\pm_{i,r}\mid i\in\{0,1,\cdots,m+n-1\},\ r\geq2\}$ are generated inductively by $\{h_{i,r},x^\pm_{i,r}\mid i\in\{0,1,\cdots,m+n-1\},\ r=0,1\}$. The following theorem describes the presentation of the affine super Yangian $Y_{\ve_1,\ve_2}(\widehat{\mathfrak{sl}}(m|n))$ whose generators are $x_{i,r}^{+}, x_{i,r}^{-}, h_{i,r}$ $(i \in \{0,1,\cdots, m+n-1\}, r = 0,1)$. 
\begin{Theorem}[Ueda~\cite{U2}, Theorem~3.13]\label{Mini}
Suppose that $m, n\geq2$ and $m\neq n$. The affine super Yangian $Y_{\ve_1,\ve_2}(\widehat{\mathfrak{sl}}(m|n))$ is isomorphic to the associative superalgebra generated by $x_{i,r}^{+}, x_{i,r}^{-}, h_{i,r}$ $(i \in \{0,1,\cdots, m+n-1\}, r = 0,1)$ subject to the following defining relations:
\begin{gather}
[h_{i,r}, h_{j,s}] = 0,\label{eq2.1}\\
[x_{i,0}^{+}, x_{j,0}^{-}] = \delta_{ij} h_{i, 0},\label{eq2.2}\\
[x_{i,1}^{+}, x_{j,0}^{-}] = \delta_{ij} h_{i, 1} = [x_{i,0}^{+}, x_{j,1}^{-}],\label{eq2.3}\\
[h_{i,0}, x_{j,r}^{\pm}] = \pm a_{ij} x_{j,r}^{\pm},\label{eq2.4}\\
[\tilde{h}_{i,1}, x_{j,0}^{\pm}] = \pm a_{ij}\left(x_{j,1}^{\pm}-m_{ij}\dfrac{\varepsilon_1 - \varepsilon_2}{2} x_{j, 0}^{\pm}\right),\label{eq2.5}\\
[x_{i, 1}^{\pm}, x_{j, 0}^{\pm}] - [x_{i, 0}^{\pm}, x_{j, 1}^{\pm}] = \pm a_{ij}\dfrac{\varepsilon_1 + \varepsilon_2}{2} \{x_{i, 0}^{\pm}, x_{j, 0}^{\pm}\} - m_{ij} \dfrac{\varepsilon_1 - \varepsilon_2}{2} [x_{i, 0}^{\pm}, x_{j, 0}^{\pm}],\label{eq2.6}\\
(\ad x_{i,0}^{\pm})^{1+|a_{ij}|} (x_{j,0}^{\pm})= 0 \ (i \neq j), \label{eq2.7}\\
[x^\pm_{i,0},x^\pm_{i,0}]=0\ (i=0, m),\label{eq2.8}\\
	[[x^\pm_{i-1,0},x^\pm_{i,0}],[x^\pm_{i,0},x^\pm_{i+1,0}]]=0\ (i=0, m),\label{eq2.9}
\end{gather}
where the generators $x^\pm_{m, r}$ and $x^\pm_{0, r}$ are odd and all other generators are even and we define $x^\pm_{-1,0}$ as $x^\pm_{m+n-1,0}$.
\end{Theorem}
There exists another presentation of the affine super Yangian. 
\begin{Proposition}\label{Prop32}
Suppose that $m, n\geq2$ and $m\neq n$. The affine super Yangian $Y_{\ve_1,\ve_2}(\widehat{\mathfrak{sl}}(m|n))$ is isomorphic to the associative superalgebra generated by $X_{i,r}^{+}, X_{i,r}^{-}, H_{i,r}$ $(i \in \{0,1,\cdots, m+n-1\}, r = 0,1)$ subject to the following defining relations:
\begin{gather}
[H_{i,r}, H_{j,s}] = 0,\label{Eq2.1}\\
[X_{i,0}^{+}, X_{j,0}^{-}] = \delta_{ij} H_{i, 0},\label{Eq2.2}\\
[X_{i,1}^{+}, X_{j,0}^{-}] = \delta_{ij} H_{i, 1} = [X_{i,0}^{+}, X_{j,1}^{-}],\label{Eq2.3}\\
[H_{i,0}, X_{j,r}^{\pm}] = \pm a_{ij} X_{j,r}^{\pm},\label{Eq2.4}\\
[\tilde{H}_{i,1}, X_{j,0}^{\pm}] = \pm a_{ij}\left(X_{j,1}^{\pm}\right),\text{ if }(i,j)\neq(0,m+n-1),(m+n-1,0),\label{Eq2.5}\\
[\tilde{H}_{0,1}, X_{m+n-1,0}^{\pm}] = \mp{(-1)}^{p(m+n)} \left(X_{m+n-1,1}^{\pm}-(\ve+\dfrac{m-n}{2}\hbar) X_{m+n-1, 0}^{\pm}\right),\label{Eq2.6}\\
[\tilde{H}_{m+n-1,1}, X_{0,0}^{\pm}] = \mp{(-1)}^{p(m+n)} \left(X_{0,1}^{\pm}+(\ve+\dfrac{m-n}{2}\hbar) X_{0, 0}^{\pm}\right),\label{Eq2.7}\\
[X_{i, 1}^{\pm}, X_{j, 0}^{\pm}] - [X_{i, 0}^{\pm}, X_{j, 1}^{\pm}] = \pm a_{ij}\dfrac{\hbar}{2} \{X_{i, 0}^{\pm}, X_{j, 0}^{\pm}\}\text{ if }(i,j)\neq(0,m+n-1),(m+n-1,0),\label{Eq2.8}\\
[X_{0, 1}^{\pm}, X_{m+n-1, 0}^{\pm}] - [X_{0, 0}^{\pm}, X_{m+n-1, 1}^{\pm}]\qquad\qquad\qquad\qquad\qquad\qquad\nonumber\\
\qquad\qquad\qquad\qquad\qquad\qquad= \pm{(-1)}^{p(m+n)}\dfrac{\hbar}{2} \{X_{0, 0}^{\pm}, X_{m+n-1, 0}^{\pm}\} - (\ve+\dfrac{m-n}{2}\hbar) [X_{0, 0}^{\pm}, X_{m+n-1, 0}^{\pm}],\label{Eq2.9}\\
(\ad X_{i,0}^{\pm})^{1+|a_{ij}|} (X_{j,0}^{\pm})= 0 \ \ (i \neq j), \label{Eq2.10}\\
[X^\pm_{i,0},X^\pm_{i,0}]=0\ (i=0, m),\label{Eq2.11}\\
[[X^\pm_{i-1,0},X^\pm_{i,0}],[X^\pm_{i,0},X^\pm_{i+1,0}]]=0\ (i=0, m),\label{Eq2.12}
\end{gather}
where $\hbar=\ve_1+\ve_2$, $\tilde{H}_{i,1}=H_{i,1}-\dfrac{\hbar}{2}H_{i,0}^2$, $\ve=-(m-n)\ve_2$, the generators $X^\pm_{m, r}$ and $X^\pm_{0, r}$ are odd and all other generators are even and we define $X^\pm_{-1,0}$ as $X^\pm_{m+n-1,0}$.
\end{Proposition}
\begin{proof}
The homomorphism $\Psi$ from $Y_{\ve_1,\ve_2}(\widehat{\mathfrak{sl}}(m|n))$ to the superalgebra defined in Proposition~\ref{Prop32} is given by
\begin{gather*}
\Psi(h_{i,0})=H_{i,0},\quad\Psi(x^\pm_{i,0})=X^\pm_{i,0},\\
\Psi(h_{i,1})=\begin{cases}
H_{0,1}&\text{ if $i=0$},\\
H_{i,1}-\dfrac{i-2\delta(i>m)(i-m)}{2}(\ve_1-\ve_2)H_{i,0}&\text{ if $i\neq0$},
\end{cases}
\end{gather*}
where 
\begin{equation*}
\delta(i>m)=\begin{cases}
1&\text{ if }i>m,\\
0&\text{ if }i\leq m.
\end{cases}
\end{equation*}
It is clear that $\Psi$ is an isomorphism.
\end{proof}
By using Theorem~\ref{Mini}, we can construct a non-trivial homomorphism from the affine super Yangian to the completion of the universal enveloping algebra of the affinization of $\mathfrak{gl}(m|n)$. The homomorphism is called as the evaluation map. 
In this paper, we deal with two different affinizations of $\mathfrak{gl}(m|n)$ corresponding to different cocycles. The first one is denoted by $\widehat{\mathfrak{gl}}(m|n)^\text{str}$ and is defined as $\mathfrak{gl}(m|n)\otimes\mathbb{C}[t,t^{-1}]\oplus\mathbb{C}\tilde{c}\oplus\mathbb{C}z$ whose commutator relations are given by
\begin{gather*}
[x\otimes t^u, y\otimes t^v]=\begin{cases}
[x,y]\otimes t^{u+v}+\delta_{u+v,0}u\text{str}(xy)\tilde{c}\ \text{ if }x,y\in\mathfrak{sl}(m|n),\\
[e_{a,b},e_{i,i}]\otimes t^{u+v}+\delta_{u+v,0}u\text{str}(E_{a,b}E_{i,i})\tilde{c}+\delta_{u+v,0}\delta_{a,b}u{(-1)}^{p(a)+p(i)}z\\
\qquad\qquad\qquad\qquad\qquad\qquad\qquad\qquad\qquad\text{ if }x=e_{a,b},\ y=e_{i,i},
\end{cases}\\
\text{$z$ and $\tilde{c}$ are central elements of }\widehat{\mathfrak{gl}}(m|n),
\end{gather*}
where $E_{i,j}\in\mathfrak{gl}(m|n)$ is a matrix unit whose parity is $p(i)+p(j)$ and str is a supertrace of $\mathfrak{gl}(m|n)$, that is, $\text{str}(E_{i,j}E_{k,l})=\delta_{i,l}\delta_{j,k}{(-1)}^{p(i)}$. 
The another one is denoted by $\widehat{\mathfrak{gl}}(m|n)^\kappa$ and is defined as $\mathfrak{gl}(m|n)\otimes\mathbb{C}[t^{\pm1}]\oplus\mathbb{C}\tilde{c}\oplus\mathbb{C}x$ whose commutator relations are
\begin{gather*}
\text{$\tilde{c}$ and $x$ are central elements},\\
[u\otimes t^a,v\otimes t^b]=[u,v]\otimes t^{a+b}+\delta_{a+b,0}a\str(uv)\tilde{c},\text{ if }u\text{ or }v\in\mathfrak{sl}(m|n)\\
[E_{i,i}\otimes t^a,E_{j,j}\otimes t^b]=\delta_{a+b,0}a\str(E_{i,i}E_{j,j})\tilde{c}-\delta_{a+b,0}al(lc-1){(-1)}^{p(i)+p(j)}x.
\end{gather*}
Next, we introduce a completion of $U(\widehat{\mathfrak{gl}}(m|n)^\text{str})/U(\widehat{\mathfrak{gl}}(m|n)^\text{str})(z-1)$ following \cite{MNT} and \cite{GNW}. For all $s\in\mathbb{Z}$, we denote $E_{i,j}\otimes t^s$ by $E_{i,j}(s)$. 
We also set the grading of $U(\widehat{\mathfrak{gl}}(m|n)^\text{str})/U(\widehat{\mathfrak{gl}}(m|n)^\text{str})(z-1)$ as $\text{deg}(X(s))=s$ and $\text{deg}(c)=0$.
Then, $U(\widehat{\mathfrak{gl}}(m|n)^\text{str})/U(\widehat{\mathfrak{gl}}(m|n)^\text{str})(z-1)$ becomes a graded algebra and we denote the set of the degree $d$ elements of $U(\widehat{\mathfrak{gl}}(m|n)^\text{str})/U(\widehat{\mathfrak{gl}}(m|n)^\text{str})(z-1)$ by $U(\widehat{\mathfrak{gl}}(m|n)^\text{str})_d$. 
We obtain the completion
\begin{equation*}
U(\widehat{\mathfrak{gl}}(m|n)^\text{str})_{{\rm comp}}=\bigoplus_{d\in\mathbb{Z}}U(\widehat{\mathfrak{gl}}(m|n)^\text{str})_{{\rm comp},d},
\end{equation*}
where
\begin{equation*}
U(\widehat{\mathfrak{gl}}(m|n)^\text{str})_{d}=\plim[N]U(\widehat{\mathfrak{gl}}(m|n)^\text{str})_{d}/\sum_{r>N}\limits U(\widehat{\mathfrak{gl}}(m|n)^\text{str})_{d-r}U(\widehat{\mathfrak{gl}}(m|n)^\text{str})_{r}.
\end{equation*}

Now, we can define the evaluation map. Let us denote
\begin{gather*}
 \hbar=\ve_1+\ve_2,\quad
\delta(i\leq j)=\begin{cases}
 1&(i\leq j),\\
 0&(i>j),
 \end{cases}\\
h_i=\begin{cases}
{(-1)}^{p(m+n)}E_{m+n,m+n}-E_{1,1}+\tilde{c}&(i=0),\\
{(-1)}^{p(i)}E_{ii}-{(-1)}^{p(i+1)}E_{i+1,i+1}&(1\leq i\leq m+n-1),
\end{cases}\\
x^+_i=\begin{cases}
E_{m+n,1}\otimes t&(i=0),\\
E_{i,i+1}&(\text{otherwise}),
\end{cases}
\quad x^-_i=\begin{cases}
{(-1)}^{p(m+n)}E_{1,m+n}\otimes t^{-1}&(i=0),\\
{(-1)}^{p(i)}E_{i+1,i}&(\text{otherwise}).
\end{cases}
\end{gather*}
\begin{Theorem}[Ueda~\cite{U2}, Proposition~5.2]\label{thm:main}
Set $\tilde{c} =\dfrac{( -m+n) \ve_1}{\hbar}$.
Then, there exists an algebra homomorphism 
\begin{equation*}
\ev_{\ve_1,\ve_2} \colon Y_{\ve_1,\ve_2}(\widehat{\mathfrak{sl}}(m|n)) \to U(\widehat{\mathfrak{gl}}(m|n)^\text{str})_{{\rm comp}}
\end{equation*}
uniquely determined by 
\begin{gather*}
	\ev_{\ve_1,\ve_2}(X_{i,0}^{+}) = x_{i}^{+}, \quad \ev_{\ve_1,\ve_2}(X_{i,0}^{-}) = x_{i}^{-},\quad \ev_{\ve_1,\ve_2}(H_{i,0}) = h_{i},
\end{gather*}
\begin{gather*}
	\ev_{\ve_1,\ve_2}(H_{i,1}) = \begin{cases}
		\hbar\tilde{c}h_{0} -{(-1)}^{p(m+n)} \hbar E_{m+n,m+n} (E_{1,1}-\tilde{c}) \\
		\ +{(-1)}^{p(m+n)}\hbar \displaystyle\sum_{s \geq 0} \limits\displaystyle\sum_{k=1}^{m+n}\limits{(-1)}^{p(k)}E_{m+n,k}(-s) E_{k,m+n}(s)\\
		\quad-\hbar \displaystyle\sum_{s \geq 0} \quad\displaystyle\sum_{k=1}^{m+n}\limits{(-1)}^{p(k)}E_{1,k}(-s-1) E_{k,1}(s+1)\\
		 \qquad\qquad\qquad\qquad\qquad\qquad\qquad\qquad\qquad\qquad\qquad\qquad\qquad\qquad\text{ if $i = 0$},\\
\\
		-\dfrac{(i-2\delta(i\geq m+1)(i-m))}{2}\hbar h_{i} -{(-1)}^{p(E_{i,i+1})} \hbar E_{i,i}E_{i+1,i+1} \\
		\ + \hbar{(-1)}^{p(i)} \displaystyle\sum_{s \geq 0}  \limits\displaystyle\sum_{k=1}^{i}\limits{(-1)}^{p(k)} E_{i,k}(-s) E_{k,i}(s)\\
		\quad +\hbar{(-1)}^{p(i)} \displaystyle\sum_{s \geq 0} \limits\displaystyle\sum_{k=i+1}^{m+n}\limits {(-1)}^{p(k)}E_{i,k}(-s-1) E_{k,i}(s+1) \\
		\qquad -\hbar{(-1)}^{p(i+1)}\displaystyle\sum_{s \geq 0}\limits\displaystyle\sum_{k=1}^{i}\limits{(-1)}^{p(k)}E_{i+1,k}(-s) E_{k,i+1}(s)\\
		 \qquad\quad-\hbar{(-1)}^{p(i+1)}\displaystyle\sum_{s \geq 0}\limits\displaystyle\sum_{k=i+1}^{m+n} \limits{(-1)}^{p(k)}E_{i+1,k}(-s-1) E_{k,i+1}(s+1)\\
		\qquad\qquad\qquad\qquad\qquad\qquad\qquad\qquad\qquad\qquad\qquad\qquad\qquad\qquad \text{ if $i \neq 0$},
	\end{cases}
\end{gather*}
\begin{align*}
\ev_{\ve_1,\ve_2}(X^+_{i,1})&=\begin{cases}
\hbar\tilde{c} x_{0}^{+} + \hbar \displaystyle\sum_{s \geq 0} \limits\displaystyle\sum_{k=1}^{m+n}\limits {(-1)}^{p(k)}E_{m+n,k}(-s) E_{k,1}(s+1)\\
\qquad\qquad\qquad\qquad\qquad\qquad\qquad\qquad\qquad\qquad\qquad\qquad\qquad\qquad \text{ if $i = 0$},\\
-\dfrac{i-2\delta(i\geq m+1)(i-m)}{2}\hbar x_{i}^{+}+ \hbar \displaystyle\sum_{s \geq 0}\limits\displaystyle\sum_{k=1}^i\limits {(-1)}^{p(k)}E_{i,k}(-s) E_{k,i+1}(s)\\
\quad+\hbar \displaystyle\sum_{s \geq 0}\limits\displaystyle\sum_{k=i+1}^{m+n}\limits {(-1)}^{p(k)}E_{i,k}(-s-1) E_{k,i+1}(s+1)\\
\qquad\qquad\qquad\qquad\qquad\qquad\qquad\qquad\qquad\qquad\qquad\qquad\qquad\qquad \text{ if $i \neq 0$},
\end{cases}
\end{align*}
\begin{align*}
\ev_{\ve_1,\ve_2}(X^-_{i,1})&=\begin{cases}
\hbar\tilde{c}x_{0}^{-} +{(-1)}^{p(m+n)} \hbar \displaystyle\sum_{s \geq 0} \limits\displaystyle\sum_{k=1}^{m+n}\limits {(-1)}^{p(k)}E_{1,k}(-s-1) E_{k,m+n}(s),\\
\qquad\qquad\qquad\qquad\qquad\qquad\qquad\qquad\qquad\qquad\qquad\qquad\qquad\qquad \text{ if $i = 0$},\\
-\dfrac{i-2\delta(i\geq m+1)(i-m)}{2}\hbar x_{i}^{-}+ {(-1)}^{p(i)}\hbar \displaystyle\sum_{s \geq 0}\limits\displaystyle\sum_{k=1}^i\limits {(-1)}^{p(k)}E_{i+1,k}(-s) E_{k,i}(s)\\
\qquad+ {(-1)}^{p(i)}\hbar \displaystyle\sum_{s \geq 0}\limits\displaystyle\sum_{k=i+1}^{m+n}\limits {(-1)}^{p(k)}E_{i+1,k}(-s-1) E_{k,i}(s+1)\ \text{ if $i \neq 0$}.
\end{cases}
\end{align*}
\end{Theorem}
It was shown in \cite{U3} that the image of $\ev_{\ve_1,\ve_2}$ is dense in $U(\widehat{\mathfrak{gl}}(m|n)^\text{str})_{{\rm comp}}$ in the case when $\ve_1\neq0$. 
\begin{Remark}
In \cite{U2}, the evaluation map was defined in terms of the generators $h_{i,r}$ and $x^\pm_{i,r}$\ $(r=0,1)$. 
\end{Remark}
In the non-super case, the affine Yangian was defined in Definition~3.2 of \cite{Gu2} and Definition~2.3 of \cite{Gu1} as follows.
\begin{Definition}
Suppose that $m\geq3$ and set two $m\times m$-matrices $(a_{i,j})$ and $(m_{i,j})$ as
\begin{gather*}
a_{ij} =
	\begin{cases}
	2  &\text{if } i=j, \\
	-1 &\text{if } i=j \pm 1, \\
	        -1 &\text{if }(i,j)=(0,m-1),(m-1,0),\\
		0  &\text{otherwise,}
	\end{cases}\ 
	 m_{i,j}=
	\begin{cases}
	1&\text{if } i=j - 1,\\
		-1 &\text{if } i=j + 1,\\
	        1 &\text{if }(i,j)=(0,m-1),\\
		-1 &\text{if }(i,j)=(m-1,0),\\
		0  &\text{otherwise}.
	\end{cases}
\end{gather*}
The affine Yangian $Y_{\ve_1,\ve_2}(\widehat{\mathfrak{sl}}(m))$ is the associative algebra over $\mathbb{C}$ generated by $x_{i,r}^{+}, x_{i,r}^{-}$, $h_{i,r}$ $(i \in \{0,1,\cdots, m-1\}, r \in \mathbb{Z}_{\geq 0})$ with parameters $\ve_1, \ve_2 \in \mathbb{C}$ subject to the defining relations \eqref{eq1.1}-\eqref{eq1.6}.
\end{Definition}
Similarly to Proposition~\ref{Prop32}, the affine Yangian $Y_{\ve_1,\ve_2}(\widehat{\mathfrak{sl}}(m))$ also has a presentation whose generators are $H_{i,r},X^\pm_{i,r}\ (0\leq i\leq m-1,\ r=0,1)$.
\begin{Proposition}
The affine Yangian $Y_{\ve_1,\ve_2}(\widehat{\mathfrak{sl}}(m))$ is isomorphic to the associative algebra generated by $X_{i,r}^{+}, X_{i,r}^{-}, H_{i,r}$ $(i \in \{0,1,\cdots, m-1\}, r = 0,1)$ subject to the defining relations \eqref{Eq2.1}-\eqref{Eq2.4}, \eqref{Eq2.10} and
\begin{gather}
[\tilde{H}_{i,1}, X_{j,0}^{\pm}] = \pm a_{ij}\left(X_{j,1}^{\pm}\right),\text{ if }(i,j)\neq(0,m-1),(m-1,0),\\
[\tilde{H}_{0,1}, X_{m-1,0}^{\pm}] = \mp\left(X_{m-1,1}^{\pm}-(\ve+\dfrac{m}{2}\hbar) X_{m-1, 0}^{\pm}\right),\\
[\tilde{H}_{m-1,1}, X_{0,0}^{\pm}] = \mp\left(X_{0,1}^{\pm}+(\ve+\dfrac{m}{2}\hbar) X_{0, 0}^{\pm}\right),\\
[X_{i, 1}^{\pm}, X_{j, 0}^{\pm}] - [X_{i, 0}^{\pm}, X_{j, 1}^{\pm}] = \pm a_{ij}\dfrac{\hbar}{2} \{X_{i, 0}^{\pm}, X_{j, 0}^{\pm}\}\text{ if }(i,j)\neq(0,m-1),(m-1,0),\\
\begin{align}
&[X_{0, 1}^{\pm}, X_{m-1, 0}^{\pm}] - [X_{0, 0}^{\pm}, X_{m-1, 1}^{\pm}]\nonumber\\
&\qquad\qquad= \pm\dfrac{\hbar}{2} \{X_{0, 0}^{\pm}, X_{m-1, 0}^{\pm}\} - (\ve+\dfrac{m}{2}\hbar) [X_{0, 0}^{\pm}, X_{m-1, 0}^{\pm}],
\end{align}
\end{gather}
where $\hbar=\ve_1+\ve_2$, $\tilde{H}_{i,1}=H_{i,1}-\dfrac{\hbar}{2}H_{i,0}^2$, and $\ve=-m\ve_2$.
\end{Proposition}
The evaluation map for the affine Yangian $Y_{\ve_1,\ve_2}(\widehat{\mathfrak{sl}}(m))$ was constructed in Section~6 of \cite{Gu1} and Theorem~3.8 of \cite{K1}. In fact, the evaluation map of \cite{Gu1} and \cite{K1} was defined in the same formula as that of Theorem~\ref{thm:main} by setting $n=0$ and assuming all of the parity is equal to zero. In the non-super case, the surjectivity of the evaluation map was shown in Theorem~4.18 of \cite{K2}.

\section{Generators of rectangular $W$-superalgebras of type $A$}
We fix some notations for vertex algebras. For a vertex algebra $V$, we denote the generating field associated with $v\in V$ by $v(z)=\displaystyle\sum_{n\in\mathbb{Z}}\limits v_{(n)}z^{-n-1}$. We also denote the OPE of $V$ by
\begin{equation*}
u(z)v(w)\sim\displaystyle\sum_{s\geq0}\limits \dfrac{(u_{(s)}v)(w)}{(z-w)^{s+1}}
\end{equation*}
for all $u, v\in V$. We denote the identity vector (resp.\ the translation operator) by $|0\rangle$ (resp.\ $\partial$).

First, we recall the definition of rectangular $W$-superalgebras of type $A$ (see \cite{KW1}, \cite{KW2}, and \cite{A}). Let us set 
\begin{equation*}
\mathfrak{g}=\mathfrak{gl}(ml|nl)=\displaystyle\bigoplus_{\substack{1\leq i,j\leq m+n\\1\leq s,t\leq l}}\limits \mathbb{C}e_{(s-1)(m+n)+i,(t-1)(m+n)+j},
\end{equation*}
where $e_{(s-1)(m+n)+i,(t-1)(m+n)+j}$ is the unit matrix whose parity is $p(i)+p(j)$.
Since $\mathfrak{gl}(ml|nl)$ is isomorphic to $\mathfrak{gl}(m|n)\otimes \mathfrak{gl}(l)$ as a graded vector space, we identify $e_{(s-1)(m+n)+i,(t-1)(m+n)+j}\in\mathfrak{gl}(ml|nl)$ with $e_{i,j}\otimes e_{s,t}\in\mathfrak{gl}(m|n)\otimes \mathfrak{gl}(l)$. We set a parity of $e_{i,j}\in\mathfrak{gl}(m|n)$ as $p(i)+p(j)$.
We take an even nilpotent element $f=\displaystyle\sum_{s=1}^{l-1}\displaystyle\sum_{i=1}^{m+n}e_{s(m+n)+i,(s-1)(m+n)+i}\in\mathfrak{gl}(ml|nl)$ and fix $k\in\mathbb{C}$. We also take $( \  | \ )$ as a supersymmetric invariant inner product of $\mathfrak{g}$ such that 
\begin{equation}\label{leq}
(u|v)=\begin{cases}
k\str(uv)&\text{ if $u$ or $v$ is an element of $\mathfrak{sl}(ml|nl)$},\\
k\str(uv)+{(-1)}^{p(i)+p(j)}(1-c)&\text{ if $u=e_{i,i}\otimes e_{r_1,r_1}, v=e_{j,j}\otimes e_{r_2,r_2}$},
\end{cases}
\end{equation}
where $c$ is a complex number and $\str$ is a supertrace of $\mathfrak{gl}(ml|nl)$. 
We set 
\begin{equation*}
\mathfrak{g}_t=\bigoplus_{\substack{1\leq i,j\leq m+n\\0\leq s\leq l-1\\0\leq s+t\leq l-1}}\mathbb{C}e_{s(m+n)+i,(s+t)(m+n)+j}.
\end{equation*}
and fix a $\mathfrak{sl}_2$-triple $(x, e, f)$ such that
\begin{equation*}
\mathfrak{g}_t=\{y\in\mathfrak{g}\mid[x,y]=ty\}.
\end{equation*}
Let us set
\begin{gather*}
S=\{(i,j,s,t)\mid1\leq i,j\leq m+n,\ 0\leq s,s+t\leq l-1\},\\
S_+=\{(i,j,s,t)\mid1\leq i,j\leq m+n,\  0\leq s,s+t\leq l-1,t\geq1\}.
\end{gather*}
For all $\beta=(i,j,s,t)\in S$, we also set $u_\beta$ as $e_{s(m+n)+i,(s+t)(m+n)+j}$ and $p(\beta)$ as the parity of $u_\beta$. Then, we have
\begin{equation*}
\mathfrak{g}=\bigoplus_{\beta\in S}\limits\mathbb{C}u_\beta,\qquad\mathfrak{g}_{\geq0}=\bigoplus_{t\geq0}\limits\mathfrak{g}_t=\bigoplus_{\beta\in S_+}\limits\mathbb{C}u_\beta.
\end{equation*}
Moreover, let $\mathfrak{b}$ be $\displaystyle\bigoplus_{j\leq0}\limits \mathfrak{g}_j$, which is a subalgebra of $\mathfrak{g}$.
We define $\kappa$ as an inner product of $\mathfrak{b}$ such that
\begin{gather*}
\kappa(u,v)=(u|v)+\frac{1}{2}\big(\kappa_\mathfrak{g}(u,v)-\kappa_{\mathfrak{g}_0}(p_0(u),p_0(v))\big)\ \text{for all }u,v\in\mathfrak{b},
\end{gather*}
where $p_0\colon\mathfrak{b}\to\mathfrak{g}_0$ is the projection map and $\kappa_\mathfrak{g}$\  $(\text{resp.}\  \kappa_{\mathfrak{g}_0})$ is the Killing form on $\mathfrak{g}$\ $(\text{resp.}\ \mathfrak{g}_0)$. By the definition of $\kappa$, we have
\begin{align*}
&\quad\kappa(e_{s_1(m+n)+i_1,t_1(m+n)+j_1},e_{s_2(m+n)+i_2,t_2(m+n)+j_2})\\
&=\delta_{s_1,t_2}\delta_{t_1,s_2}\delta_{i_1,j_2}\delta_{j_1,i_2}{(-1)}^{p(i_1)}(k+(l-1)(m-n))\\
&\qquad\qquad-\delta_{s_1,t_1}\delta_{s_2,t_2}\delta_{i_1,j_1}\delta_{i_2,j_2}{(-1)}^{p(i_1)+p(i_2)}(c-\delta_{s_1,s_2}).
\end{align*}
Let $\hat{\mathfrak{b}}$ be the Lie superalgebra $\mathfrak{b}\otimes\mathbb{C}[t^{\pm1}]\oplus\mathbb{C}y$ whose commutator relations are
\begin{gather*}
[at^u,bt^v]=[a,b]t^{u+v}+\delta_{u+v,0}u\kappa(a,b)y,\\
\text{$y$ is a central element}.
\end{gather*}
We also set a left $\hat{\mathfrak{b}}$-module $V^\kappa(\mathfrak{b})$ as $U(\hat{\mathfrak{b}})/U(\hat{\mathfrak{b}})(\mathfrak{b}[t]\oplus\mathbb{C}(y-1))\cong U(\mathfrak{b}[t^{-1}]t^{-1})$. 
Then, it has a vertex algebra structure whose identity vector is $1$ and the generating field $(ut^{-1})(z)$ is equal to $\displaystyle\sum_{s\in\mathbb{Z}}\limits(ut^s)z^{-s-1}$ for all $u\in\mathfrak{b}$. We call $V^\kappa(\mathfrak{b})$ the universal affine vertex algebra associated with $(\mathfrak{b},\kappa)$.

In order to simplify the notation, we denote the generating field $(ut^{-1})(z)$ as $u(z)$. By the definition of $V^\kappa(\mathfrak{b})$, generating fields $u(z)$ and $v(z)$ satisfy the OPE
\begin{gather}
u(z)v(w)\sim\dfrac{[u,v](w)}{z-w}+\dfrac{\kappa(u,v)}{(z-w)^2}\label{OPE1}
\end{gather}
for all $u,v\in\mathfrak{b}$. 

We set a Lie superalgebra $\mathfrak{a}_{m,n}=\displaystyle\bigoplus_{u_\beta\in\mathfrak{g}_{\leq0}}\limits\mathbb{C}J^{(u_\beta)}\oplus\displaystyle\bigoplus_{u_\beta\in\mathfrak{g}_{<0}}\limits\mathbb{C}\psi_{(u_\beta)}$  with the following commutator relations;
\begin{gather*}
[J^{(u)},J^{(v)}]=J^{([u,v])},\quad [J^{(e_{i,j})},\psi_{e_{s,t}}]=\delta_{j,s}\psi_{e_{i,t}}-\delta_{i,t}{(-1)}^{p(e_{i,j})(p(e_{s,t})+1)}\psi_{e_{s,j}},\quad [\psi_u,\psi_v]=0,
\end{gather*}
where the parity of $J^{(u_\beta)}$ (resp.\ $\psi_{u_\beta}$) is equal to $p(\beta)$ (resp.\ $p(\beta)+1$) and we denote $\displaystyle\sum_{u_\beta\in\mathfrak{g}_{\leq0}}\limits a_\beta J^{(u_\beta)}$ (resp.\ $\displaystyle\sum_{u_\beta\in\mathfrak{g}_{<0}}\limits a_\beta \psi_{(u_\beta)}$) by $J^{(\sum_{u_\beta\in\mathfrak{g}_{\leq0}} a_\beta u_\beta)}$ (resp.\ $\psi_{\sum_{u_\beta\in\mathfrak{g}_{<0}} a_\beta u_\beta}$) for all $a_\beta\in\mathbb{C}$. We define an affinization of $\mathfrak{a}_{m,n}$ by using the inner product on $\mathfrak{a}_{m,n}$ such that
\begin{gather*}
\kappa_{m,n}(J^{(u)},J^{(v)})=\kappa(u,v),\quad \kappa_{m,n}(J^{(u)},\psi_v)=\kappa_{m,n}(\psi_u,\psi_v)=0.
\end{gather*}
By \eqref{OPE1}, $V^{\kappa_{m,n}}(\mathfrak{a}_{m,n})$ contains $V^\kappa(\mathfrak{b})$. We identify $ut^{-1}\in V^\kappa(\mathfrak{b})$ with $J^{(u)}t^{-1}\in V^{\kappa_{m,n}}(\mathfrak{a}_{m,n})$.

For all $u\in \mathfrak{a}_{m,n}$, let $u[-s]$ be $ut^{-s}$. In this section, we regard $V^{\kappa_{m,n}}(\mathfrak{a}_{m,n})$ (resp.\ $V^\kappa(\mathfrak{b})$) as a non-associative superalgebra whose product $\cdot$ is defined by
\begin{equation*}
u[-t]\cdot v[-s]=(u[-t])_{(-1)}v[-s].
\end{equation*}
We sometimes omit $\cdot$ and denote $\psi_{e_{(v+w)(m+n)+j,v(m+n)+i}}[s]$ by $\psi_{(v+w)(m+n)+j,v(m+n)+i}[s]$ in order to simplify the notation. A rectangular $W$-superalgebra $\mathcal{W}^k(\mathfrak{gl}(lm|ln),(l^{(m|n)}))$ can be realized as the subalgebra of $V^{\kappa_{m,n}}(\mathfrak{a}_{m,n})$ (\cite{KW1} and \cite{KW2}) as follows.

Let us set $\alpha$ as $k+(l-1)(m-n)$. We can define an odd differential $d_0 \colon V^{\kappa}(\mathfrak{b})\to V^{\kappa_{m,n}}(\mathfrak{a}_{m,n})$ determined by
\begin{gather}
d_01=0,\\
[d_0,\partial]=0,
\end{gather}
\begin{align}
&\quad[d_0,e_{(s-1)(m+n)+j,(t-1)(m+n)+i}[-1]]\nonumber\\
&=\sum_{\substack{t< a\leq s,\\1\leq r\leq m+n}}{(-1)}^{p(e_{i,j})+p(e_{i,r})p(e_{r,j})}e_{(a-1)(m+n)+r, (t-1)(m+n)+i}[-1]\psi_{(s-1)(m+n)+j,(a-1)(m+n)+r}[-1]\nonumber\\
&\quad-\sum_{\substack{t\leq a< s,\\1\leq r\leq m+n}}{(-1)}^{p(e_{i,r})p(e_{r,j})}\psi_{(a-1)(m+n)+r,(t-1)(m+n)+i}[-1]e_{(s-1)(m+n)+j, (a-1)(m+n)+r}[-1]\nonumber\\
&\quad +\delta(s<t){(-1)}^{p(j)}\alpha\psi_{(s-1)(m+n)+j,(t-1)(m+n)+i}[-2]\nonumber\\
&\quad +{(-1)}^{p(j)}\psi_{s(m+n)+j,(t-1)(m+n)+i}[-1]-\psi_{(s-1)(m+n)+j,(t-2)(m+n)+i}[-1].\label{ee1}
\end{align}
\begin{Definition}[Kac-Roan-Wakimoto~\cite{KRW}, Theorem~2.4]\label{T125}
The rectangular $W$-superalgebra associated with a Lie superalgebra $\mathfrak{gl}(m|n)$ and a nilpotent element $f=\displaystyle\sum_{s=1}^{l-1}\displaystyle\sum_{i=1}^{m+n}e_{s(m+n)+i,(s-1)(m+n)+i}$ is the vertex subalgebra defined by
\begin{equation*}
\mathcal{W}^k(\mathfrak{gl}(ml|nl),(l^{(m|n)}))=\{y\in V^\kappa(\mathfrak{b})\subset V^{\kappa_{m,n}}(\mathfrak{a}_{m,n})\mid d_0(y)=0\}.
\end{equation*}
\end{Definition}
We denote the rectangular $W$-superalgebra associated with a Lie superalgebra $\mathfrak{gl}(m|n)$ and a nilpotent element $f$ by $\mathcal{W}^k(\mathfrak{gl}(lm|ln),(l^{(m|n)}))$. The rest of this section is devoted to the construction of two kinds of elements $W^{(1)}_{i,j}$ and $W^{(2)}_{i,j}$, which are generators of $\mathcal{W}^k(\mathfrak{gl}(lm|ln),(l^{(m|n)}))$. 

We regard $V^\kappa(\mathfrak{b})\otimes\mathbb{C}[\tau]$ and $V^{\kappa_{m,n}}(\mathfrak{a}_{m,n})\otimes\mathbb{C}[\tau]$ as non-associative superalgebras whose defining relations are given by
\begin{equation*}
u[-t]\cdot v[-s]=(u[-t])_{(-1)}v[-s],\ [\tau,u[-s]]=su[-s],
\end{equation*}
where $\tau$ is an even element.
Let $\widetilde{d}_0^{m,n}\colon V^{\kappa_{m,n}}(\mathfrak{a}_{m,n})\otimes\mathbb{C}[\tau]\to V^{\kappa_{m,n}}(\mathfrak{a}_{m,n})\otimes\mathbb{C}[\tau]$ be the odd differential determined by
\begin{equation*}
\widetilde{d}_0^{m,n}1=0,\quad [\widetilde{d}_0^{m,n},u[-s]]=[d_0,u[-s]],\quad [\widetilde{d}_0^{m,n},\tau]=0.
\end{equation*}
First, let us recall how to construct generators of the principal $W$-algebra $\mathcal{W}^k(\mathfrak{gl}(l),(l^1))$ (\cite{AM}, Section~2). We denote by $T(C)$ a non-associative free algebra associated with a vector space $C$ and by $\mathfrak{gl}(l)_{\leq0}$ the Lie algebra $\displaystyle\bigoplus_{\substack{1\leq j\leq i\leq l}}\limits\mathbb{C}e_{i,j}$. In the principal case, $\mathfrak{b}$ is equal to $\mathfrak{gl}(l)_{\leq0}$. 
By Definition~\ref{T125}, the principal $W$-algebra can be defined as
\begin{equation*}
\mathcal{W}^k(\mathfrak{gl}(l),(l^1))=\{x\in V^\kappa(\mathfrak{gl}(l)_{\leq0})\otimes\mathbb{C}[\tau]\mid d_0(x)=0\}.
\end{equation*}
Similarly to $V^\kappa(\mathfrak{b})\otimes\mathbb{C}[\tau]$, we define a non-associative algebra $T(\mathfrak{gl}(l)_{\leq0}[t^{-1}]t^{-1})\otimes \mathbb{C}[\tau]$. Let us set $\pi$ as $k+l-1$ and an $l\times l$ matrix $B=(b_{i,j})_{1\leq i,j\leq l}$ as
\begin{equation}
\begin{bmatrix}
\pi\tau+e_{1,1}[-1] &-1\phantom{-}&0&\dots & 0\\[0.4em]
e_{2,1}[-1] &\pi\tss\tau+e_{2,2}[-1] &-1\phantom{-}&\dots & 0\\[0.4em]
\vdots &\vdots &\ddots & &\vdots\\[0.4em]
e_{l-1,1}[-1] &e_{l-1,2}[-1] &\dots
&\pi\tss\tau+e_{l-1,l-1}[-1] &-1\phantom{-}\\[0.4em]
e_{l, 1}[-1] &e_{l, 2}[-1] &\dots &e_{l,l-1}[-1]  &\pi\tss\tau+e_{l,l}[-1]
\end{bmatrix}\label{matrx}
\end{equation}
whose entries are elements of $T(\mathfrak{gl}(l)_{\leq0}[t^{-1}]t^{-1})\otimes \mathbb{C}[\tau]$. 
For any matrix $A=(a_{i,j})_{1\leq i,j\leq s}$, we define $\text{cdet}(A)$ as 
\begin{equation*}
\displaystyle\sum_{\sigma\in\mathfrak{S}_s}\limits\text{sgn}(\sigma)a_{\sigma(1),1}\big(a_{\sigma(2),2}(a_{\sigma(3),3}\cdots a_{\sigma(s-1),s-1})a_{\sigma(s),s}\big)\in T(\mathfrak{gl}(l)_{\leq0}[t^{-1}]t^{-1})\otimes\mathbb{C}[\tau].
\end{equation*}
By the commutator relation of $T(\mathfrak{gl}(l)_{\leq0}[t^{-1}]t^{-1})\otimes\mathbb{C}[\tau]$, we can rewrite $\text{cdet}(B)$ as $\displaystyle\sum_{r=0}^l\limits\widetilde{W}^{(r)}(\pi\tau)^{l-r}$ such that $\widetilde{W}^{(r)}\in T(\mathfrak{gl}(l)_{\leq0}[t^{-1}]t^{-1})$. Let $p$ be the projection map from $T(\mathfrak{gl}(l)_{\leq0}[t^{-1}]t^{-1})$ to $V^\kappa(\mathfrak{gl}(l)_{\leq0})=U(\mathfrak{gl}(l)_{\leq0}[t^{-1}]t^{-1})$ and $W^{(r)}$ be $p(\widetilde{W}^{(r)})$.
Proving that $[\widetilde{d}_0^{1,0},p(\text{cdet}(B))]=0$, we obtain the following theorem (see Theorem 2.1 of \cite{AM}).
\begin{Theorem}\label{T305}
The $W$-superalgebra $\mathcal{W}^k(\mathfrak{gl}(l),(l^1))$ is generated by $\{W^{(r)}\}_{1\leq r\leq l}$.
\end{Theorem}
\begin{Remark}
In \cite{AM}, the tensor algebra $T(C)$ should have been defined as a non-associative superalgebra as above since $V(\mathfrak{g}_{\leq0})$ is non-associative.
\end{Remark}
Let $A_{1,0}$ be a quotient algebra of $T(\mathfrak{a}_{1,0}[t^{-1}]t^{-1})\otimes\mathbb{C}[\tau]$ subjected to the relation
\begin{equation*}
(e_{a,1}[-1]\psi_{i,a}[-1])\text{cdet}(C^{l-i})-e_{a,1}[-1](\psi_{i,a}[-1]\text{cdet}(C^{l-i}))=0\text{ for all }1\leq a\leq i,
\end{equation*}
where $C^{l-i}$ is a submatrix of $B$ consisting of the last $(l-i)$ rows and columns.
Constructing a homomorphism
\begin{equation*}
D\colon T(\mathfrak{gl}(l)_{\leq0}[t^{-1}]t^{-1})\otimes\mathbb{C}[\tau]\to A_{1,0}
\end{equation*}
determined by
\begin{align*}
D(e_{s,u}[-1])&=\sum_{\substack{u< a\leq s}}e_{a, u}[-1]\psi_{s,a}[-1]-\sum_{\substack{u\leq a<s}}\psi_{a,u}[-1]e_{s, a}[-1]\\
&\quad+\delta(s<u)\pi\psi_{s,u}[-2]+\psi_{s,u+1}[-1]-\psi_{s-1,u}[-1],
\end{align*}
we obtain the relation $D(\text{cdet}(B))=0$ in the way similar to the one of Theorem~2.1 of \cite{AM}.

We regard $\mathfrak{gl}(m|n)$ as an associative superalgebra whose product $\cdot$ is determined by $e_{i,j}\cdot e_{s,u}=\delta_{j,s}e_{i,u}$.
Then, we obtain a non-associative superalgebra $\mathfrak{gl}(m|n)\otimes V^\kappa(\mathfrak{b})\otimes\mathbb{C}[\tau]$.
We construct a homomorphism 
\begin{equation*}
T\colon T(\mathfrak{gl}(l)_{\leq0}[t^{-1}]t^{-1})\otimes\mathbb{C}[\tau]\to \mathfrak{gl}(m|n)\otimes V^\kappa(\mathfrak{b})\otimes\mathbb{C}[\tau]
\end{equation*}
determined by
\begin{gather*}
T_{i,j}(x)={(-1)}^{p(i)}x\otimes e_{i,j}\in\mathfrak{gl}(l)_{\leq0}[t^{-1}]t^{-1}\otimes \mathfrak{gl}(m|n)=\mathfrak{b}[t^{-1}]t^{-1},\quad T(\tau)=\tau,
\end{gather*}
where $T_{i,j}(x)$ is defined as $e_{j,i}\otimes T_{i,j}(x)=T(x)$. Since $T$ is a homomorphism, we obtain
\begin{gather*}
T_{i,j}(xy)=\sum_{r=1}^{m+n}{(-1)}^{p(e_{i,r})p(e_{j,r})}T_{r,i}(x)T_{j,r}(y).
\end{gather*}
By the commutator relation of $V^\kappa(\mathfrak{b})$ and $\mathbb{C}[\tau]$, $W^{(r)}_{i,j}\in V^\kappa(\mathfrak{b})$ is defined by
\begin{equation}
T_{j,i}(\text{cdet}(B))=\sum_{r=0}^l{(-1)}^{p(j)}W^{(r)}_{i,j}(\alpha\tau)^{l-r},\label{5261}
\end{equation}
where $B$ is defined by replacing $\pi$ in \eqref{matrx} with $\alpha$. 
\begin{Theorem}\label{T306}
For all $m,n\geq 0$ such that $m\neq n$, the $W$-superalgebra $\mathcal{W}^k(\mathfrak{gl}(ml|nl),(l^{(m|n)}))$ is freely generated by $\{W^{(r)}_{i,j}\mid1\leq r\leq l,1\leq i,j\leq m+n\}$.
\end{Theorem}
\begin{Remark}
In the case when $n=0$, Theorem~\ref{T306} is shown in Theorem~3.1 of \cite{AM}. 
\end{Remark}
\begin{proof}
Under the assumption that $\pi$ is equal to $\alpha$, we denote $A_{1,0}$ (resp.\ $D$) as $\bar{A}_{1,0}$ (resp.\ $\bar{D}$).
We construct a homomorphism $T^p\colon \bar{A}_{1,0}\to \mathfrak{gl}(m|n)\otimes V^{\kappa_{m,n}}(\mathfrak{a}_{m,n})\otimes \mathbb{C}[\tau]$
determined by
\begin{gather*}
T^p_{i,j}(e_{s,w}[u])={(-1)}^{p(j)}e_{(s-1)(m+n)+i,(w-1)(m+n)+j}[u],\\
T^p_{i,j}(\psi_{s,w}[u])=\psi_{(s-1)(m+n)+i,(w-1)(m+n)+j}[u],\quad T^p(\tau)=\tau,
\end{gather*}
where $T^p_{i,j}(x)$ is defined as $e_{j,i}\otimes T^p_{i,j}(x)=T^p(x)$. Since $T^p$ is a homomorphism, we obtain
\begin{gather*}
T^p_{i,j}(e_{s,w}[-1]\psi_{u,v}[-1])=\sum_{r=1}^{m+n}{(-1)}^{p(e_{i,r})p(e_{j,r})}T^p_{r,j}(e_{s,w}[-1])T^p_{i,r}(\psi_{u,v}[-1]),\\
T^p_{i,j}(\psi_{u,v}[-1]e_{s,w}[-1])=\sum_{r=1}^{m+n}{(-1)}^{p(e_{i,r})+p(e_{i,r})p(e_{j,r})}T^p_{r,j}(\psi_{u,v}[-1])T^p_{i,r}(e_{s,w}[-1]).
\end{gather*}
By the definition of $T_{j,i}$ and $d_0$, we have
\begin{align}
&\quad[\widetilde{d}_0^{m,n},T_{j,i}(e_{s,w})]\nonumber\\
&=[\widetilde{d}_0^{m,n},{(-1)}^{p(j)}e_{(s-1)(m+n)+j,(w-1)(m+n)+i}[-1]]\nonumber\\
&=\sum_{\substack{w< a\leq s,\\1\leq r\leq m+n}}{(-1)}^{p(i)+p(e_{i,r})p(e_{j,r})}e_{(a-1)(m+n)+r, (w-1)(m+n)+i}[-1]\psi_{(s-1)(m+n)+j,(a-1)(m+n)+r}[-1]\nonumber\\
&\quad-\sum_{\substack{w\leq a<s,\\1\leq r\leq m+n}}{(-1)}^{\gamma}\psi_{(a-1)(m+n)+r,(w-1)(m+n)+i}[-1]e_{(s-1)(m+n)+j, (a-1)(m+n)+r}[-1]\nonumber\\
&\quad+\delta(s<w)\alpha\psi_{(s-1)(m+n)+j,(w-1)(m+n)+i}[-2]\nonumber\\
&\quad+\psi_{s(m+n)+j,(w-1)(m+n)+i}[-1]-\psi_{(s-1)(m+n)+j,(w-2)(m+n)+i}[-1]\nonumber\\
&=T_{j,i}^p([\bar{D},e_{s,w}]),
\end{align}
where $\gamma=p(j)+p(e_{i,r})p(e_{j,r})$. 
Thus, the relation $[\widetilde{d}_0^{m,n},T_{j,i}(a)]=T_{j,i}^p([\bar{D},a])$ holds for all $a\in T(\mathfrak{gl}(l)_{\leq0})$.
Then, we obtain
\begin{equation}\label{2212}
[\widetilde{d}_0^{m,n},T_{j,i}(\text{cdet}(B))]=T_{j,i}^p([\bar{D},\text{cdet}(B)]).
\end{equation}
Since $[\bar{D},\text{cdet}(B)]=0$ holds by the proof of Theorem~2.1 of \cite{AM}, the right hand side of \eqref{2212} is equal to zero. Thus, we have obtained the relation $[d_0,W^{(r)}_{i,j}]=0$. The rest of the proof is same as \cite{AM}.
\end{proof}
In particular, by \eqref{5261}, we have
\begin{align}
W^{(1)}_{i,j}&=\displaystyle\sum_{1\leq s\leq l}\limits e_{(s-1)(m+n)+j,(s-1)(m+n)+i}[-1],\label{W1}\\
W^{(2)}_{i,j}&=\displaystyle\sum_{1\leq s\leq l-1}\limits e_{s(m+n)+j,(s-1)(m+n)+i}[-1]+\alpha\displaystyle\sum_{1\leq s\leq l}\limits (s-1)e_{(s-1)(m+n)+j,(s-1)(m+n)+i}[-2]\nonumber\\
&\quad+\displaystyle\sum_{\substack{r_1<r_2\\1\leq t\leq m+n}}\limits{(-1)}^{p(t)+p(e_{i,t})p(e_{j,t})} e^{(r_1)}_{t,i}[-1]e^{(r_2)}_{j,t}[-1],\label{W2}
\end{align}
where we set $e^{(r)}_{j,i}$ as $e_{(r-1)(m+n)+j,(r-1)(m+n)+i}$.
\begin{Theorem}\label{Tinf}
The rectangular $W$-superalgebra $\mathcal{W}^{k}(\mathfrak{gl}(ml|nl),(l^{(m|n)}))$ is generated by $W^{(1)}_{i,j}$ and $W^{(2)}_{i,j}$ $(1\leq i,j \leq m +n)$ provided that $\alpha=k+(l-1)(m-n)\neq0$, $m\neq n$ and $m+n\geq2$.
\end{Theorem}
Theorem~\ref{Tinf} is proved in the appendix.
\begin{Remark}
In the case when $(m,n)=(1,0)$ or $(0,1)$, the elements $W^{(1)}_{i,i+1}$ or $W^{(2)}_{i,i+1}$ do not exist. This is the reason why we need the condition that $m+n\geq2$ in Theorem~\ref{Tinf}.
\end{Remark}
\section{OPEs of rectangular $W$-superalgebras}
First, let us recall the definition of the universal enveloping algebras of vertex algebras.
For all vertex algebra $V$, let $L(V)$ be the Borchards Lie algebra, that is,
\begin{align}
 L(V)=V{\otimes}\mathbb{C}[t,t^{-1}]/\text{Im}(\partial\otimes\id +\id\otimes\frac{d}{d t})\label{844},
\end{align}
where the commutation relation is given by
\begin{align*}
 [ut^a,vt^b]=\sum_{r\geq 0}\begin{pmatrix} a\\r\end{pmatrix}(u_{(r)}v)t^{a+b-r}
\end{align*}
for all $u,v\in V$ and $a,b\in \mathbb{Z}$. Now, we define the universal enveloping algebra of $V$.
\begin{Definition}[Frenkel-Zhu~\cite{FZ}, Matsuo-Nagatomo-Tsuchiya~\cite{MNT}]\label{Defi}
We set $\mathcal{U}(V)$ as the quotient algebra of the standard degreewise completion of the universal enveloping algebra of $L(V)$ by the completion of the two-sided ideal generated by
\begin{gather}
(u_{(a)}v)t^b-\sum_{i\geq 0}
\begin{pmatrix}
 a\\i
\end{pmatrix}
(-1)^i(ut^{a-i}vt^{b+i}-{(-1)}^{p(u)p(v)}(-1)^avt^{a+b-i}ut^{i}),\label{241}\\
|0\rangle t^{-1}-1,
\end{gather}
where $|0\rangle$ is the identity vector of $V$.
We call $\mathcal{U}(V)$ the universal enveloping algebra of $V$.
\end{Definition}
\begin{Lemma}[Kac-Roan-Wakimoto~\cite{KRW}, Theorem~2.4]\label{Lem1}
There exists a homomorphism from the universal enveloping algebra of $\widehat{\mathfrak{gl}}(m|n)^\kappa$ to $\mathcal{U}(\mathcal{W}^k(\mathfrak{gl}(ml|nl), (l^{(m|n)})))$ determined by
\begin{gather*}
\xi(E_{i,j}t^s)=W^{(1)}_{j,i}t^s,\quad\xi(\tilde{c})=l\alpha t^{-1},\quad\xi(x)=1.
\end{gather*}
\end{Lemma}
In order to construct a homomorphism from the affine super Yangian to the universal enveloping algebra of $W$-superalgebras in Section~6, we need to compute the following terms; 
\begin{gather*}
(W^{(1)}_{i,j})_{(u)}W^{(2)}_{s,t}\ (u\geq0),\quad(W^{(2)}_{i,i})_{(0)}W^{(2)}_{j,j},\quad(W^{(2)}_{i,i})_{(1)}W^{(2)}_{j,j}.
\end{gather*}
First, we compute $(W^{(1)}_{i,j})_{(u)}W^{(2)}_{s,t}\ (u\geq0)$. By direct computation, we obtain the below two lemmas. We only give the proof of the first equation of Lemma~\ref{Lem3}.
\begin{Lemma}\label{Lem2}
We obtain
\begin{gather*}
(W^{(1)}_{u,v})_{(0)}W^{(2)}_{i,j}=\delta_{j,u}W^{(2)}_{i,v}-\delta_{i,v}{(-1)}^{p(e_{u,v})p(e_{i,j})}W^{(2)}_{u,j}.
\end{gather*}
\end{Lemma}
\begin{Lemma}\label{Lem3}
The following equations hold;
\begin{gather*}
(W^{(1)}_{v,w})_{(1)}W^{(2)}_{i,j}=\delta_{j,v}(l-1)\alpha W^{(1)}_{i,w}-\delta_{v,w}{(-1)}^{p(w)}(l-1)(lc-1)W^{(1)}_{i,j},\\
(W^{(1)}_{v,w})_{(2)}W^{(2)}_{i,j}=l(l-1)\alpha\kappa(e_{w,v},e_{j,i}),\\
(W^{(1)}_{v,w})_{(s)}W^{(2)}_{i,j}=0\ (\text{for all }s\geq3).
\end{gather*}
\end{Lemma}
\begin{proof}[\textbf{proof of the first equation of Lemma~\ref{Lem3}}]
By the definition of $W^{(1)}_{v,w}$ and $W^{(2)}_{i,j}$, we obtain
\begin{align}
&\quad(W^{(1)}_{v,w})_{(1)}W^{(2)}_{i,j}\nonumber\\
&=\sum_{\substack{1\leq r_1<r_2\leq l\\1\leq t\leq m+n\\1\leq s\leq l}}{(-1)}^{p(i)+p(e_{i,j})p(e_{i,t})}\kappa(e^{(s)}_{w,v},e^{(r_1)}_{t,i})e^{(r_2)}_{j,t}[-1]\nonumber\\
&\quad+\sum_{\substack{1\leq r_1<r_2\leq l\\1\leq t\leq m+n\\1\leq s\leq l}}{(-1)}^{p(i)+p(e_{i,j})p(e_{i,t})+p(e_{w,v})p(e_{i,t})}\kappa(e^{(s)}_{w,v},e^{(r_2)}_{j,t})e^{(r_1)}_{t,i}[-1]\nonumber\\
&\quad+\alpha\sum_{1\leq s,t\leq l}(t-1)[e^{(s)}_{w,v},e^{(t)}_{j,i}][-1].\label{5240}
\end{align}
Let us compute each terms of the right hand side of \eqref{5240}. Since we obtain
\begin{equation*}
\kappa(e^{(s)}_{w,v},e^{(r_1)}_{t,i})=\delta_{r_1,s}\delta_{i,w}\delta_{v,t}{(-1)}^{p(i)}\alpha+(c-\delta_{r_1,s})\delta_{v,w}\delta_{t,i}{(-1)}^{p(w)+p(i)}
\end{equation*}
the first term of \eqref{5240} is equal to
\begin{align}
&\sum_{\substack{1\leq r_1<r_2\leq l}}{(-1)}^{p(e_{i,j})p(e_{i,v})}\delta_{i,w}\alpha e^{(r_2)}_{j,v}[-1]\nonumber\\
&\qquad-\sum_{\substack{1\leq r_1<r_2\leq l\\1\leq s\leq l}}{(-1)}^{p(w)+p(e_{i,j})p(e_{i,i})}(c-\delta_{r_1,s})\delta_{v,w}e^{(r_2)}_{j,i}[-1].\label{5249}
\end{align}
By a direct computation, we obtain
\begin{align*}
&\quad\text{the first term of \eqref{5249}}\\
&=\sum_{\substack{1\leq r_2\leq l}}(r_2-1){(-1)}^{p(e_{i,j})p(e_{i,v})}\delta_{i,w}\alpha e^{(r_2)}_{j,v}[-1],\\
&\quad\text{the second term of \eqref{5249}}\\
&=\sum_{\substack{1\leq r_2\leq l}}{(-1)}^{p(w)}(r_2-1)\delta_{v,w}lce^{(r_2)}_{j,i}+\sum_{\substack{1\leq r_2\leq l}}{(-1)}^{p(w)}(r_2-1)\delta_{v,w}e^{(r_2)}_{j,i}[-1]\\
&=\sum_{\substack{1\leq r_2\leq l}}{(-1)}^{p(w)}(r_2-1)\delta_{v,w}(lc-1)e^{(r_2)}_{j,i}[-1].
\end{align*}
Thus, we obtain
\begin{align}
&\quad\text{the first term of \eqref{5241}}\nonumber\\
&=\sum_{\substack{1\leq r_2\leq l}}(r_2-1){(-1)}^{p(e_{i,j})p(e_{i,v})}\delta_{i,w}\alpha e^{(r_2)}_{j,v}[-1]-\sum_{\substack{1\leq r_2\leq l}}{(-1)}^{p(w)}(r_2-1)\delta_{v,w}(lc-1)e^{(r_2)}_{j,i}[-1].\label{5241}
\end{align}
Similarly to the first term, we find that the second term of \eqref{5240} is equal to
\begin{align}
&\sum_{\substack{1\leq r_1<r_2\leq l}}{(-1)}^{p(i)+p(e_{i,j})p(e_{i,w})+p(e_{w,i})p(e_{w,v})}\delta_{j,v}\alpha{(-1)}^{p(w)}e^{(r_1)}_{w,i}[-1]\nonumber\\
&\qquad+\sum_{\substack{1\leq r_1<r_2\leq l}}{(-1)}^{p(i)+p(e_{i,j})p(e_{i,j})+p(e_{j,i})p(e_{v,w})}\delta_{v,w}(c-\delta_{r_1,s}){(-1)}^{p(w)+p(j)}e^{(r_1)}_{j,i}[-1]\nonumber\\
&=\sum_{\substack{1\leq r_1\leq l}}\alpha(l-r_1)\delta_{j,v}e^{(r_1)}_{w,i}[-1]-\sum_{\substack{1\leq r_1\leq l}}{(-1)}^{p(w)}(l-r_1)\delta_{v,w}(lc-1)e^{(r_1)}_{j,i}[-1].\label{5242}
\end{align}
By a direct computation, the third term of \eqref{5240} is equal to 
\begin{align}
\sum_{1\leq t\leq l}\alpha(t-1)\delta_{j,v}e^{(t)}_{w,i}-\alpha\sum_{1\leq t\leq l}(t-1)\delta_{w,i}{(-1)}^{p(e_{v,w})p(e_{i,j})}e^{(t)}_{j,v}.\label{5243}
\end{align}
Adding \eqref{5240}, \eqref{5241}, and \eqref{5242}, we have obtained the statement.
\end{proof}
\begin{Corollary}\label{COR}
The following equation holds;
\begin{align*}
&\quad[W^{(1)}_{v,w}t^s, W^{(2)}_{i,j}t^u]\\
&=\delta_{j,v}W^{(2)}_{i,w}t^{s+u}-\delta_{i,w}{(-1)}^{p(e_{v,w})p(e_{i,j})}W^{(2)}_{v,j}t^{s+u}\\
&\quad+\delta_{j,v}s(l-1)\alpha W^{(1)}_{i,w}t^{s+u-1}-\delta_{v,w}{(-1)}^{p(w)}(l-1)(lc-1)sW^{(1)}_{i,j}t^{s+u-1}\\
&\quad+\dfrac{s(s-1)}{2}l(l-1)\alpha\kappa(e_{w,v},e_{j,i})t^{s+u-2}.
\end{align*}
\end{Corollary}
The following assertion is also shown by direct calculation. We omit the proof.
\begin{Lemma}\label{Lem4}
We obtain
\begin{align*}
(W^{(2)}_{i,i})_{(0)}W^{(2)}_{j,j}&={(-1)}^{p(i)}(W^{(1)}_{i,j})_{(-1)}W^{(2)}_{j,i}-{(-1)}^{p(j)}(W^{(1)}_{j,i})_{(-1)}W^{(2)}_{i,j}-(\delta_{i,j}\alpha+{(-1)}^{p(i)})\partial W^{(2)}_{j,j}\nonumber\\
&\quad+{(-1)}^{p(j)}(l-1)\alpha(W^{(1)}_{j,i})_{(-1)}\partial W^{(1)}_{i,j}-\{(l-1)^2c-(l-1)\}(W^{(1)}_{j,j})_{(-1)}\partial W^{(1)}_{i,i}\\
&\quad+\delta_{i,j}\dfrac{l(l-1)}{2}\alpha^2\partial^2W^{(1)}_{i,i}+{(-1)}^{p(j)}\dfrac{l(l-1)}{2}\alpha\partial^2W^{(1)}_{i,i}-{(-1)}^{p(j)}\dfrac{l(l-1)^2}{2}c\alpha\partial^2W^{(1)}_{i,i}\\
&\quad+\frac{1}{2}{(-1)}^{p(i)}(l-1)\alpha\partial^2W^{(1)}_{j,j}-\frac{1}{2}{(-1)}^{p(j)}(l-1)\alpha\partial^2W^{(1)}_{i,i}
\end{align*}
and
\begin{align*}
(W^{(2)}_{i,i})_{(1)}W^{(2)}_{j,j}&=-\{(l-1)^2c-(l-1)\}(W^{(1)}_{j,j})_{(-1)}W^{(1)}_{i,i}-2\delta_{i,j}\alpha W^{(2)}_{i,i}-{(-1)}^{p(i)} W^{(2)}_{j,j}\\
&\quad-{(-1)}^{p(j)} W^{(2)}_{i,i}+{(-1)}^{p(j)}(l-1)\alpha(W^{(1)}_{j,i})_{(-1)}W^{(1)}_{i,j}+\delta_{i,j}l(l-1)\alpha^2\partial W^{(1)}_{i,i}\\
&\quad+{(-1)}^{p(j)}l(l-1)\alpha\partial W^{(1)}_{i,i}-{(-1)}^{p(j)}l(l-1)^2c\alpha\partial W^{(1)}_{i,i}\\
&\quad+{(-1)}^{p(i)}(l-1)\alpha\partial W^{(1)}_{j,j}-{(-1)}^{p(j)}(l-1)\alpha\partial W^{(1)}_{i,i}.
\end{align*}
\end{Lemma}
\begin{Remark}
In \cite{Rap}, Rap\v{c}\'{a}k defined two kinds of elements of rectangular $W$-superalgebras of type $A$, which are called $U_{1,i,j}$ and $U_{2,i,j}\ (1\leq i,j\leq m+n)$ under the assumption that $c=0$. The element $U_{r,i,j}$ is corresponding to ${(-1)}^{p(i)p(j)}W^{(r)}_{i,j}\ (r=1,2)$, where $J^{(x)}_{a,b}$ in \cite{Rap} is corresponding to ${(-1)}^{p(a)p(b)}e_{b,a}$ in this paper. 
\end{Remark}
\section{Affine Super Yangians and Rectangular $W$-superalgebras}
In this section, we prove the main result of this paper.
Here after, we assume that $m\neq n$ and set
\begin{gather*}
\ve_1=\dfrac{\alpha}{m-n},\quad\ve_2=-1-\dfrac{\alpha}{m-n}
\end{gather*}
and fix an invariant inner product on $\mathfrak{gl}(m|n)$ such that $c=0$ (see \eqref{leq}). 
\begin{Theorem}\label{Main}
There exists an algebra homomorphism 
\begin{equation*}
\Phi\colon Y_{\ve_1,\ve_2}(\widehat{\mathfrak{sl}}(m|n))\to \mathcal{U}(\mathcal{W}^{k}(\mathfrak{gl}(ml|nl),(l^{(m|n)})))
\end{equation*} 
determined by
\begin{align*}
\Phi(H_{i,0})=\begin{cases}
{(-1)}^{p(m+n)}W^{(1)}_{m+n,m+n}-W^{(1)}_{1,1}+l\alpha&(i=0),\\
{(-1)}^{p(i)}W^{(1)}_{i,i}-{(-1)}^{p(i+1)}W^{(1)}_{i+1,i+1}&(i\neq 0),
\end{cases}\\
\Phi(X^+_{i,0})=\begin{cases}
W^{(1)}_{1,m+n}t&(i=0),\\
W^{(1)}_{i+1,i}&(i\neq0),
\end{cases}
\quad \Phi(X^-_{i,0})=\begin{cases}
{(-1)}^{p(m+n)}W^{(1)}_{m+n,1}t^{-1}&(i=0),\\
{(-1)}^{p(i)}W^{(1)}_{i,i+1}&(i\neq0),
\end{cases}
\end{align*}
\begin{align*}
\Phi(H_{i,1})&=\begin{cases}{(-1)}^{p(m+n)}W^{(2)}_{m+n,m+n}t-W^{(2)}_{1,1}t+{(-1)}^{p(m+n)}(l-1)\alpha W^{(1)}_{m+n,m+n}\\
\quad- l\alpha\Phi(H_{0,0}) +{(-1)}^{p(m+n)}W^{(1)}_{m+n,m+n} (W^{(1)}_{1,1}-l\alpha)\\
\quad-{(-1)}^{p(m+n)}\displaystyle\sum_{s \geq 0} \limits\displaystyle\sum_{u=1}^{m+n}\limits{(-1)}^{p(u)}W^{(1)}_{u,m+n}t^{-s} W^{(1)}_{m+n,u}t^s\\
\quad+\displaystyle\sum_{s \geq 0}\displaystyle\sum_{u=1}^{m+n}\limits{(-1)}^{p(u)}W^{(1)}_{u,1}t^{-s-1} W^{(1)}_{1,u}t^{s+1},\\
\qquad\qquad\qquad\qquad\qquad\qquad\qquad\qquad\qquad\qquad\qquad\qquad\qquad\qquad i=0,\\
{(-1)}^{p(i)}W^{(2)}_{i,i}t-{(-1)}^{p(i+1)}W^{(2)}_{i+1,i+1}t\\
\quad+\dfrac{i-2\delta(i\geq m+1)(i-m)}{2}\Phi(H_{i,0})+{(-1)}^{p(E_{i,i+1})}  W^{(1)}_{i,i}W^{(1)}_{i+1,i+1}\\
\quad-{(-1)}^{p(i)} \displaystyle\sum_{s \geq 0}  \limits\displaystyle\sum_{u=1}^{i}\limits{(-1)}^{p(u)} W^{(1)}_{u,i}t^{-s}W^{(1)}_{i,u}t^s\\
\quad-{(-1)}^{p(i)} \displaystyle\sum_{s \geq 0} \limits\displaystyle\sum_{u=i+1}^{m+n}\limits {(-1)}^{p(u)}W^{(1)}_{u,i}t^{-s-1} W^{(1)}_{i,u}t^{s+1}\\
\quad+{(-1)}^{p(i+1)}\displaystyle\sum_{s \geq 0}\limits\displaystyle\sum_{u=1}^{i}\limits{(-1)}^{p(u)}W^{(1)}_{u,i+1}t^{-s} W^{(1)}_{i+1,u}t^s\\
\quad+{(-1)}^{p(i+1)}\displaystyle\sum_{s \geq 0}\limits\displaystyle\sum_{u=i+1}^{m+n} \limits{(-1)}^{p(u)}W^{(1)}_{u,i+1}t^{-s-1} W^{(1)}_{i+1,u}t^{s+1}\\
\qquad\qquad\qquad\qquad\qquad\qquad\qquad\qquad\qquad\qquad\qquad\qquad\qquad\qquad i\neq0,
\end{cases}
\end{align*}
\begin{align*}
\Phi(X^+_{i,1})&=\begin{cases}
W^{(2)}_{1,m+n}t^2+(l-1)\alpha W^{(1)}_{1,m+n}t-
l\alpha \Phi(X_{0,0}^{+})-\displaystyle\sum_{s \geq 0} \limits\displaystyle\sum_{u=1}^{m+n}\limits {(-1)}^{p(u)}W^{(1)}_{u,m+n}t^{-s} W^{(1)}_{1,u}t^{s+1}\\
\qquad\qquad\qquad\qquad\qquad\qquad\qquad\qquad\qquad\qquad\qquad\qquad\qquad\qquad \text{ if $i = 0$},\\
W^{(2)}_{i+1,i}t+\dfrac{i-2\delta(i\geq m+1)(i-m)}{2}\Phi(X_{i,0}^{+})\\
\quad-\displaystyle\sum_{s \geq 0}\limits\displaystyle\sum_{u=1}^i\limits {(-1)}^{p(u)}W^{(1)}_{u,i}t^{-s} W^{(1)}_{i+1,u}t^s-\displaystyle\sum_{s \geq 0}\limits\displaystyle\sum_{u=i+1}^{m+n}\limits {(-1)}^{p(u)}W^{(1)}_{u,i}t^{-s-1} W^{(1)}_{i+1,u}t^{s+1}\\
\qquad\qquad\qquad\qquad\qquad\qquad\qquad\qquad\qquad\qquad\qquad\qquad\qquad\qquad \text{ if $i \neq 0$},
\end{cases}
\end{align*}
\begin{align*}
\Phi(X^-_{i,1})&=\begin{cases}
{(-1)}^{p(m+n)}W^{(2)}_{m+n,1}-l\alpha\Phi(X_{0,0}^{-})\\
\quad-{(-1)}^{p(m+n)}\displaystyle\sum_{s \geq 0} \limits\displaystyle\sum_{u=1}^{m+n}\limits {(-1)}^{p(u)}W^{(1)}_{1,u}t^{-s-1} W^{(1)}_{m+n,u}t^s,\\
\qquad\qquad\qquad\qquad\qquad\qquad\qquad\qquad\qquad\qquad\qquad\qquad\qquad\qquad \text{ if $i = 0$},\\
{(-1)}^{p(i)}W^{(2)}_{i,i+1}t+\dfrac{i-2\delta(i\geq m+1)(i-m)}{2}\Phi(X_{i,0}^{-})\\
\quad-{(-1)}^{p(i)}\displaystyle\sum_{s \geq 0}\limits\displaystyle\sum_{u=1}^i\limits {(-1)}^{p(u)}W^{(1)}_{u,i+1}t^{-s} W^{(1)}_{i,u}t^s\\
\quad-{(-1)}^{p(i)}\displaystyle\sum_{s \geq 0}\limits\displaystyle\sum_{u=i+1}^{m+n}\limits {(-1)}^{p(u)}W^{(1)}_{u,i+1}t^{-s-1} W^{(1)}_{i,u}t^{s+1} \\
\qquad\qquad\qquad\qquad\qquad\qquad\qquad\qquad\qquad\qquad\qquad\qquad\qquad\qquad\text{ if $i \neq 0$}.
\end{cases}
\end{align*}
\end{Theorem}
\begin{proof}
It is enough to show that $\Phi$ is compatible with the defining relations \eqref{Eq2.1}-\eqref{Eq2.12}. By Lemma~\ref{Lem1}, we find that $\Phi$ is compatible with \eqref{Eq2.2}, \eqref{Eq2.10}, \eqref{Eq2.11} and \eqref{Eq2.12}. 
Thus, it is enough to show that $\Phi$ is compatible with \eqref{Eq2.1} and \eqref{Eq2.3}-\eqref{Eq2.9}. We divide the proof into two pieces, that is, Claim~\ref{Claim1234} and Claim~\ref{Claim1236} below. In Claim~\ref{Claim1234}, we show that $\Phi$ is compatible with \eqref{Eq2.3}-\eqref{Eq2.9}. In Claim~\ref{Claim1236}, we prove that $\Phi$ is compatible with \eqref{Eq2.1}.

In order to prove Claims~\ref{Claim1234} and \ref{Claim1236}, we relate $\Phi$ with the evaluation map of the affine super Yangian. We set $\widetilde{\ev}(H_{i,s})$ and $\widetilde{\ev}(X^\pm_{i,s})$ $(s=0,1)$ as
\begin{align*}
\widetilde{\ev}(H_{i,0})=\Phi(H_{i,0}),\quad\widetilde{\ev}(X^\pm_{i,0})=\Phi(X^\pm_{i,0}),
\end{align*}
\begin{align*}
\widetilde{\ev}(H_{i,1})&=\begin{cases}
\Phi(H_{0,1})-\{{(-1)}^{p(m+n)}W^{(2)}_{m+n,m+n}t-W^{(2)}_{1,1}t+{(-1)}^{p(m+n)}(l-1)\alpha W^{(1)}_{m+n,m+n}\}\\
\qquad\qquad\qquad\qquad\qquad\qquad\qquad\qquad\qquad\qquad\qquad\qquad\qquad\qquad\text{ if }i=0,\\
\Phi(H_{i,1})-\{{(-1)}^{p(i)}W^{(2)}_{i,i}t-{(-1)}^{p(i+1)}W^{(2)}_{i+1,i+1}t\}\text{ if }i\neq0,
\end{cases}
\end{align*}
\begin{align*}
\widetilde{\ev}(X^+_{i,1})&=\begin{cases}
\Phi(X^+_{i,1})-\{W^{(2)}_{1,m+n}t^2+(l-1)\alpha W^{(1)}_{1,m+n}t\}&\text{ if }i=0,\\
\Phi(X^+_{i,1})-W^{(2)}_{i+1,i}t&\text{ if }i\neq0,
\end{cases}
\end{align*}
\begin{align*}
\widetilde{\ev}(X^-_{i,1})&=\begin{cases}
\Phi(X^-_{i,1})-{(-1)}^{p(m+n)}W^{(2)}_{m+n,1}&\text{ if }i=0,\\
\Phi(X^-_{i,1})-{(-1)}^{p(i)}W^{(2)}_{i,i+1}t&\text{ if }i\neq 0.
\end{cases}
\end{align*}
We note that $\widehat{\mathfrak{gl}}(m|n)^\kappa$ is the same as $\widehat{\mathfrak{gl}}(m|n)^\text{str}$ except of the inner product on the diagonal part. By Lemma~\ref{Lem1}, we can prove that $\widetilde{\ev}$ is compatible with \eqref{Eq2.2}-\eqref{Eq2.12} which are parts of the defining relations of the affine super Yangian $Y_{\frac{l\alpha}{m-n},-1-\frac{l\alpha}{m-n}}(\widehat{\mathfrak{sl}}(m|n))$ in a way similar to the proof of the existence of the evaluation map (see Theorem~5.2 in \cite{U2}). This is summarized as the following lemma.
\begin{Lemma}\label{Claim1211}
Let us set
\begin{equation*}
\tilde{\ve}_1=\frac{l\alpha}{m-n},\qquad\tilde{\ve}_2=-1-\frac{l\alpha}{m-n}.
\end{equation*}
Then, $\widetilde{\ev}$ is compatible with \eqref{Eq2.2}-\eqref{Eq2.12} which are parts of the defining relations of the affine super Yangian $Y_{\tilde{\ve}_1,\tilde{\ve}_2}(\widehat{\mathfrak{sl}}(m|n))$.
\end{Lemma}
We remark that $\widetilde{\ev}$ is not an algebra homomorphism since $[\widetilde{\ev}(H_{i,1}),\widetilde{\ev}(H_{j,1})]$ is not equal to zero. See \eqref{equat2} below for the details.
\begin{Claim}\label{Claim1234}
For all $i,j\in \{0,1,\cdots,m+n-1\}$, $\Phi$ is compatible with \eqref{Eq2.3}-\eqref{Eq2.9}.
\end{Claim}
\begin{proof}
We only show that $\Phi$ is compatible with \eqref{Eq2.6}. The other cases are proven in a similar way. It is enough to show that
\begin{align}
&\quad[\Phi(\tilde{H}_{0,1}),\Phi(X^+_{m+n-1,0})]\nonumber\\
&=-{(-1)}^{p(m+n)}\{\Phi(X^+_{m+n-1,1})-\left((m-n)+\alpha-\dfrac{m-n}{2}\right)W^{(1)}_{m+n,m+n-1}\},\label{eqA8}\\
&\quad[\Phi(\tilde{H}_{0,1}),\Phi(X^-_{m+n-1,0})]\nonumber\\
&={(-1)}^{p(m+n)}\{\Phi(X^-_{m+n-1,1})-{(-1)}^{p(m+n-1)}\left((m-n)+\alpha-\dfrac{m-n}{2}\right)W^{(1)}_{m+n-1,m+n}\}.\label{eqA13}
\end{align}
By the definition of $\Phi$, we can rewrite the left hand side of \eqref{eqA8} as
\begin{align}
&\quad[\Phi(\tilde{H}_{0,1}),\Phi(X^+_{m+n-1,0})]\nonumber\\
&=-[W^{(1)}_{m+n,m+n-1},{(-1)}^{p(m+n)}W^{(2)}_{m+n,m+n}t]+[W^{(1)}_{m+n,m+n-1},W^{(2)}_{1,1}t]\nonumber\\
&\quad-[W^{(1)}_{m+n,m+n-1},{(-1)}^{p(m+n)}(l-1)\alpha W^{(1)}_{m+n,m+n}]+[\widetilde{\ev}(\tilde{H}_{0,1}),\widetilde{\ev}(X^+_{m+n-1,0})].\label{eqA9}
\end{align}
By Corollary~\ref{COR}, we obtain
\begin{align}
-[W^{(1)}_{m+n,m+n-1},{(-1)}^{p(m+n)}W^{(2)}_{m+n,m+n}t]&=-{(-1)}^{p(m+n)}W^{(2)}_{m+n,m+n-1}t,\label{eqA10}\\
[W^{(1)}_{m+n,m+n-1},W^{(2)}_{1,1}t]&=0.
\end{align}
By Corollary~\ref{COR}, we have
\begin{gather}
-[W^{(1)}_{m+n,m+n-1},{(-1)}^{p(m+n)}(l-1)\alpha W^{(1)}_{m+n,m+n}]=-{(-1)}^{p(m+n)}(l-1)\alpha W^{(1)}_{m+n,m+n-1}.\label{eqA11}
\end{gather}
By Lemma~\ref{Claim1211}, we also obtain
\begin{align}
&\quad[\widetilde{\ev}(\tilde{H}_{0,1}),\widetilde{\ev}(X^+_{m+n-1,0})]\nonumber\\
&=-{(-1)}^{p(m+n)}\widetilde{\ev}(X^+_{m+n-1,1})+{(-1)}^{p(m+n)}\left((m-n)+l\alpha-\dfrac{m-n}{2}\right)W^{(1)}_{m+n,m+n-1}.\label{eqA12}
\end{align}
The identity \eqref{eqA8} follows by applying \eqref{eqA10}-\eqref{eqA12} to \eqref{eqA9}. We can prove that $\Phi$ is compatible with \eqref{Eq2.7} in a similar way. 

Similarly, by the definition of $\Phi$, we obtain
\begin{align}
&\quad[\Phi(\tilde{H}_{0,1}),{(-1)}^{p(m+n-1)}W^{(1)}_{m+n-1,m+n}]\nonumber\\
&=-[{(-1)}^{p(m+n-1)}W^{(1)}_{m+n-1,m+n},{(-1)}^{p(m+n)}W^{(2)}_{m+n,m+n}t]\nonumber\\
&\quad-[{(-1)}^{p(m+n-1)}W^{(1)}_{m+n-1,m+n},-W^{(2)}_{1,1}t]\nonumber\\
&\quad-[{(-1)}^{p(m+n-1)}W^{(1)}_{m+n-1,m+n},{(-1)}^{p(m+n)}(l-1)\alpha W^{(1)}_{m+n,m+n}]\nonumber\\
&\quad+[\widetilde{\ev}(\tilde{H}_{0,1}),\widetilde{\ev}(X^-_{m+n-1,0})].\label{eqA14}
\end{align}
By Corollary~\ref{COR}, we obtain
\begin{align}
-[{(-1)}^{p(m+n-1)}W^{(1)}_{m+n-1,m+n},{(-1)}^{p(m+n)}W^{(2)}_{m+n,m+n}t]&={(-1)}^{p(m+n-1)+p(m+n)}W^{(2)}_{m+n-1,m+n}t,\label{eqA15}\\
-[{(-1)}^{p(m+n-1)}W^{(1)}_{m+n-1,m+n},W^{(2)}_{1,1}t]&=0.\label{eqA16}
\end{align}
By Lemma~\ref{Lem1}, we have
\begin{align}
&\quad-[{(-1)}^{p(m+n-1)}W^{(1)}_{m+n-1,m+n},{(-1)}^{p(m+n)}(l-1)\alpha W^{(1)}_{m+n,m+n}]\nonumber\\
&={(-1)}^{p(m+n-1)+p(m+n)}(l-1)\alpha W^{(1)}_{m+n-1,m+n}.\label{eqA17}
\end{align}
By Lemma~\ref{Claim1211}, we obtain
\begin{align}
&\quad[\widetilde{\ev}(\tilde{H}_{0,1}),\widetilde{\ev}(X^-_{m+n-1,0})]\nonumber\\
&={(-1)}^{p(m+n)}\widetilde{\ev}(X^-_{m+n-1,1})\nonumber\\
&\qquad\qquad\qquad\qquad\qquad-{(-1)}^{p(m+n)}\left((m-n)+l\alpha-\dfrac{m-n}{2}\right)\left({(-1)}^{p(m+n-1)}W^{(1)}_{m+n-1,m+n}\right).\label{eqA18}
\end{align}
The identity \eqref{eqA13} follows by applying \eqref{eqA15}-\eqref{eqA18} to \eqref{eqA14}.
Thus, we have shown that $\Phi$ is compatible with \eqref{Eq2.6}. 
\end{proof}
Finally, we prove that $\Phi$ is compatible with \eqref{Eq2.1}.
\begin{Claim}\label{Claim1236}
The following equation holds for all $i,j\in \{0,1,\cdots,m+n-1\},\ r,s\in\{0,1\}$;
\begin{equation*}
[\Phi(H_{i,r}),\Phi(H_{j,s})]=0.
\end{equation*}
\end{Claim}
\begin{proof}
By Lemma~\ref{Lem1}, we obtain $[\Phi(H_{i,0}),\Phi(H_{j,0})]=0$. In the similar way as that of Claim~\ref{Claim1234}, we have $[\Phi(H_{i,0}),\Phi(H_{j,1})]=0$. Thus, it is enough to show that $[\Phi(H_{i,1}),\Phi(H_{j,1})]=0$. 
We only show the case when $i,j\neq0$ and $i> j$. The other case is proven in a similar way. 
In order to simplify the notation, we set
\begin{align*}
X_i&=-{(-1)}^{p(i)} \displaystyle\sum_{s \geq 0}  \limits\displaystyle\sum_{u=1}^{i}\limits{(-1)}^{p(u)} W^{(1)}_{u,i}t^{-s} W^{(1)}_{i,u}t^s\\
&\quad-{(-1)}^{p(i)} \displaystyle\sum_{s \geq 0} \limits\displaystyle\sum_{u=i+1}^{m+n}\limits {(-1)}^{p(u)}W^{(1)}_{u,i}t^{-s-1}W^{(1)}_{i,u}t^{s+1}.
\end{align*}
By the definition of $\widetilde{\ev}$, we obtain
\begin{align}
\widetilde{\ev}(H_{i,1})&= \dfrac{i-2\delta(i \geq m+1)(i-m)}{2}((-1)^{p(i)}W_{i,i}^{(1)}-(-1)^{p(i+1)}W_{i+1,i+1}^{(1)}) \nonumber\\
 &\quad+(-1)^{p(E_{i,i+1})} W_{i,i}^{(1)}W_{i+1,i+1}^{(1)} + X_{i} - X_{i+1} -(W_{i+1,i+1}^{(1)})^{2}\nonumber\\
&=X_i-X_{i+1}+(\text{the term generated by }\{W^{(1)}_{i,i}t^0|1\leq i\leq m+n\}).\label{gath0}
\end{align}
By Lemma~\ref{Lem1}, Lemma~\ref{Lem2} and \eqref{gath0}, we obtain
\begin{gather}
[\widetilde{\ev}(H_{i,1}),\widetilde{\ev}(H_{j,1})]=[X_i-X_{i+1},X_j-X_{j+1}],\label{gath1}\\
[\widetilde{\ev}(H_{i,1}),({(-1)}^{p(j)}W^{(2)}_{j,j}-{(-1)}^{p(j+1)}W^{(2)}_{j+1,j+1})t]\qquad\qquad\qquad\qquad\qquad\qquad\nonumber\\
\qquad\qquad\qquad\qquad\qquad=[X_i-X_{i+1},({(-1)}^{p(j)}W^{(2)}_{j,j}-{(-1)}^{p(j+1)}W^{(2)}_{j+1,j+1})t].\label{gath2}
\end{gather}
We remark that $[\widetilde{\ev}(H_{i,1}),\widetilde{\ev}(H_{j,1})]$ is not equal to zero since the inner products on the diagonal parts of $\widehat{\mathfrak{gl}}(m|n)^\kappa$ and $\widehat{\mathfrak{gl}}(m|n)^\text{str}$ are different.

By \eqref{gath1}, \eqref{gath2}, and the definition of $\Phi$, we obtain
\begin{align*}
&\quad[\Phi(H_{i,1}),\Phi(H_{j,1})]\\
&=[({(-1)}^{p(i)}W^{(2)}_{i,i}-{(-1)}^{p(i+1)}W^{(2)}_{i+1,i+1})t,({(-1)}^{p(j)}W^{(2)}_{j,j}-{(-1)}^{p(j+1)}W^{(2)}_{j+1,j+1})t]\\
&\quad+[X_i-X_{i+1},({(-1)}^{p(j)}W^{(2)}_{j,j}-{(-1)}^{p(j+1)}W^{(2)}_{j+1,j+1})t]\\
&\quad+[({(-1)}^{p(i)}W^{(2)}_{i,i}-{(-1)}^{p(i+1)}W^{(2)}_{i+1,i+1})t,X_j-X_{j+1}]+[X_i-X_{i+1},X_j-X_{j+1}].
\end{align*}
Thus, it is enough to show the relation
\begin{align}
[{(-1)}^{p(i)}W^{(2)}_{i,i}t,{(-1)}^{p(j)}W^{(2)}_{j,j}t]+[X_i,{(-1)}^{p(j)}W^{(2)}_{j,j}t]+[{(-1)}^{p(i)}W^{(2)}_{i,i}t,X_j]+[X_i,X_j]=0\label{equat}
\end{align}
holds for all $i,j\in \{1,\cdots,m+n\},\ i\geq j$. Let us compute each terms of the left hand side of \eqref{equat}. First, we compute the first term of the left hand side of \eqref{equat}.
By Lemma~\ref{Lem4}, we obtain
\begin{align*}
&\quad{(-1)}^{p(i)+p(j)}[W^{(2)}_{i,i}t,W^{(2)}_{j,j}t]\\
&={(-1)}^{p(i)+p(j)}(W^{(2)}_{i,i})_{(0)}W^{(2)}_{j,j}t^2+{(-1)}^{p(i)+p(j)}(W^{(2)}_{i,i})_{(1)}W^{(2)}_{j,j}t\\
&={(-1)}^{p(j)}(W^{(1)}_{i,j})_{(-1)}W^{(2)}_{j,i}t^2-{(-1)}^{p(i)}(W^{(1)}_{j,i})_{(-1)}W^{(2)}_{i,j}t^2-\delta_{i,j}\alpha\partial W^{(2)}_{j,j}t^2\nonumber\\
&\quad-{(-1)}^{p(j)}\partial W^{(2)}_{j,j}t^2+{(-1)}^{p(i)}(l-1)\alpha(W^{(1)}_{j,i})_{(-1)}\partial W^{(1)}_{i,j}t^2\\
&\quad-{(-1)}^{p(i)+p(j)}\{(l-1)^2c-(l-1)\}(W^{(1)}_{j,j})_{(-1)}\partial W^{(1)}_{i,i}t^2\\
&\quad+\delta_{i,j}\dfrac{l(l-1)}{2}\alpha^2\partial^2W^{(1)}_{i,i}t^2+{(-1)}^{p(i)}\dfrac{l(l-1)}{2}\alpha\partial^2W^{(1)}_{i,i}t^2\\
&\quad-{(-1)}^{p(i)}\dfrac{l(l-1)^2}{2}c\alpha\partial^2W^{(1)}_{i,i}t^2+\frac{1}{2}{(-1)}^{p(j)}(l-1)\alpha\partial^2W^{(1)}_{j,j}t^2-\frac{1}{2}{(-1)}^{p(i)}(l-1)\alpha\partial^2W^{(1)}_{i,i}t^2\\
&\quad-{(-1)}^{p(i)+p(j)}\{(l-1)^2c-(l-1)\}(W^{(1)}_{j,j})_{(-1)}W^{(1)}_{i,i}t-2\delta_{i,j}\alpha W^{(2)}_{i,i}t\\
&\quad-{(-1)}^{p(j)} W^{(2)}_{j,j}t-{(-1)}^{p(i)} W^{(2)}_{i,i}t+{(-1)}^{p(i)}(l-1)\alpha(W^{(1)}_{j,i})_{(-1)} W^{(1)}_{i,j}t\\
&\quad+\delta_{i,j}l(l-1)\alpha^2\partial W^{(1)}_{i,i}t+{(-1)}^{p(j)}l(l-1)\alpha\partial W^{(1)}_{i,i}t-{(-1)}^{p(i)}l(l-1)^2c\alpha\partial W^{(1)}_{i,i}t\\
&\quad-{(-1)}^{p(i)}(l-1)\alpha\partial W^{(1)}_{i,i}t+{(-1)}^{p(j)}(l-1)\alpha\partial W^{(1)}_{j,j}t.
\end{align*}
We can rewrite it as
\begin{align}
&-{(-1)}^{p(i)}W^{(2)}_{i,i}t+{(-1)}^{p(j)}W^{(2)}_{j,j}t+{(-1)}^{p(j)}(W^{(1)}_{i,j})_{(-1)}W^{(2)}_{j,i}t^2-{(-1)}^{p(i)}(W^{(1)}_{j,i})_{(-1)}W^{(2)}_{i,j}t^2\nonumber\\
&\quad+{(-1)}^{p(i)}(l-1)\alpha(W^{(1)}_{j,i})_{(-1)}\partial W^{(1)}_{i,j}t^2+{(-1)}^{p(i)}(l-1)\alpha(W^{(1)}_{j,i})_{(-1)}W^{(1)}_{i,j}t\nonumber\\
&\quad-{(-1)}^{p(i)+p(j)}\{(l-1)^2c-(l-1)\}\left((W^{(1)}_{j,j})_{(-1)}\partial W^{(1)}_{i,i}t^2+(W^{(1)}_{j,j})_{(-1)}W^{(1)}_{i,i}t\right)\label{equat0}
\end{align}
since six relations
\begin{gather}
-\delta_{i,j}\alpha\partial W^{(2)}_{j,j}t^2-2\delta_{i,j}\alpha W^{(2)}_{j,j}t=0,\\
-{(-1)}^{p(i)}W^{(2)}_{i,i}t-{(-1)}^{p(j)}W^{(2)}_{j,j}t-{(-1)}^{p(j)}\partial W^{(2)}_{j,j}t^2\qquad\qquad\qquad\qquad\nonumber\\
\qquad\qquad\qquad\qquad\qquad\qquad\qquad\qquad\qquad=-({(-1)}^{p(i)}W^{(2)}_{i,i}t-{(-1)}^{p(j)}W^{(2)}_{j,j}t),\\
\delta_{i,j}\dfrac{l(l-1)}{2}\alpha^2\partial^2W^{(1)}_{i,i}t^2+\delta_{i,j}l(l-1)\alpha^2\partial W^{(1)}_{i,i}t=0,\\
{(-1)}^{p(i)}\dfrac{l(l-1)}{2}\alpha\partial^2W^{(1)}_{i,i}t^2+{(-1)}^{p(i)}l(l-1)\alpha\partial W^{(1)}_{i,i}t=0,\\
-{(-1)}^{p(i)}\dfrac{l(l-1)^2}{2}c\alpha\partial^2W^{(1)}_{i,i}t^2-{(-1)}^{p(i)}l(l-1)^2c\alpha\partial W^{(1)}_{i,i}t=0,\\
\frac{1}{2}{(-1)}^{p(j)}(l-1)\alpha\partial^2W^{(1)}_{j,j}t^2-\frac{1}{2}{(-1)}^{p(i)}(l-1)\alpha\partial^2W^{(1)}_{i,i}t^2\qquad\qquad\qquad\qquad\nonumber\\
\qquad\qquad\qquad\qquad\qquad\qquad\qquad-{(-1)}^{p(i)}(l-1)\alpha\partial W^{(1)}_{i,i}t+{(-1)}^{p(j)}(l-1)\alpha\partial W^{(1)}_{j,j}t=0
\end{gather}
hold by the definition of the translation operator $\partial$. 

In order to rewrite \eqref{equat0}, we remark that the following two relations
\begin{gather}
(x_{(-1)}y)t=\displaystyle\sum_{s\geq0}\limits xt^{-1-s}yt^{s+1}+{(-1)}^{p(x)p(y)}yt^{-s}xt^s,\label{92939}
\end{gather}
\begin{align}
(x_{(-1)}\partial y)t^2&=\displaystyle\sum_{s\geq0}\limits xt^{-1-s}(\partial y)t^{s+2}+{(-1)}^{p(x)p(y)}(\partial y)t^{1-s}xt^s\nonumber\\
&\quad=\displaystyle\sum_{s\geq0}\limits(-(s+2)xt^{-1-s}yt^{s+1}-{(-1)}^{p(x)p(y)}(1-s)yt^{-s}xt^s)\label{92940}
\end{align}
hold by \eqref{241} for all $x,y\in \mathcal{W}^k(\mathfrak{gl}(ml|nl),(l^{(m|n)}))$. By \eqref{92939} and \eqref{92940}, we also obtain
\begin{align}
(x_{(-1)}\partial y)t^2+(x_{(-1)}y)t=\sum_{s\geq0}(-(1+s)xt^{-1-s}yt^{1+s}+{(-1)}^{p(x)p(y)}syt^{-s}xt^s).\label{92941}
\end{align}
By \eqref{92939}-\eqref{92941}, we can rewrite \eqref{equat0} as
\begin{align}
&-{(-1)}^{p(i)}W^{(2)}_{i,i}t+{(-1)}^{p(j)}W^{(2)}_{j,j}t\nonumber\\
&\quad+{(-1)}^{p(j)} \displaystyle\sum_{s \geq 0}  \limits W^{(1)}_{i,j}t^{-s-1}W^{(2)}_{j,i}t^{s+2}+{(-1)}^{p(i)} \displaystyle\sum_{s \geq 0}  \limits W^{(2)}_{j,i}t^{1-s} W^{(1)}_{i,j}t^s\nonumber\\
&\quad-{(-1)}^{p(i)} \displaystyle\sum_{s \geq 0}  \limits W^{(1)}_{j,i}t^{-s-1}W^{(2)}_{i,j}t^{s+2}-{(-1)}^{p(j)} \displaystyle\sum_{s \geq 0}  \limits W^{(2)}_{i,j}t^{1-s} W^{(1)}_{j,i}t^{s}\nonumber\\
&\quad+{(-1)}^{p(j)}(l-1)\alpha \displaystyle\sum_{s \geq 0}  \limits s W^{(1)}_{i,j}t^{-s}W^{(1)}_{j,i}t^s-{(-1)}^{p(i)} (l-1)\alpha\displaystyle\sum_{s \geq 0}  \limits s W^{(1)}_{j,i}t^{-s}W^{(1)}_{i,j}t^s\nonumber\\
&\quad-{(-1)}^{p(i)+p(j)}\{(l-1)^2c-(l-1)\}\sum_{s \geq 0} (-sW^{(1)}_{j,j}t^{-s}W^{(1)}_{i,i}t^s+sW^{(1)}_{i,i}t^{-s}W^{(1)}_{j,j}t^s).\label{equat1}
\end{align}

Next, let us compute the last term of \eqref{equat}. By a computation similar to the proof of the existence of the evaluation map (see Theorem~5.2 of \cite{U2}), it is equal to
\begin{align}
[X_i,X_j]&=-{(-1)}^{p(i)+p(j)}l(lc-1)\sum_{s\geq0}s\{W^{(1)}_{i,i}t^{-s}W^{(1)}_{j,j}t^{s}-W^{(1)}_{j,j}t^{-s}W^{(1)}_{i,i}t^{s}\}\nonumber\\
&\quad-{(-1)}^{p(i)+p(j)}\sum_{s\geq0}s\{W^{(1)}_{i,i}t^{-s}W^{(1)}_{j,j}t^{s}-W^{(1)}_{j,j}t^{-s}W^{(1)}_{i,i}t^{s}\}.\label{equat2}
\end{align}
Finally, let us compute the second term and the third term of \eqref{equat}. By the definition of $X_i$, we obtain
\begin{align}
&\quad[X_i,W^{(2)}_{j,j}t]\nonumber\\
&=-{(-1)}^{p(i)} \displaystyle\sum_{s \geq 0}  \limits\displaystyle\sum_{u=1}^{i}\limits{(-1)}^{p(u)} W^{(1)}_{u,i}t^{-s} [W^{(1)}_{i,u}t^s,W^{(2)}_{j,j}t]\nonumber\\
&\quad-{(-1)}^{p(i)} \displaystyle\sum_{s \geq 0}  \limits\displaystyle\sum_{u=1}^{i}\limits{(-1)}^{p(u)} [W^{(1)}_{u,i}t^{-s},W^{(2)}_{j,j}t] W^{(1)}_{i,u}t^s\nonumber\\
&\quad-{(-1)}^{p(i)} \displaystyle\sum_{s \geq 0} \limits\displaystyle\sum_{u=i+1}^{m+n}\limits {(-1)}^{p(u)}W^{(1)}_{u,i}t^{-s-1} [W^{(1)}_{i,u}t^{s+1},W^{(2)}_{j,j}t]\nonumber\\
&\quad-{(-1)}^{p(i)} \displaystyle\sum_{s \geq 0} \limits\displaystyle\sum_{u=i+1}^{m+n}\limits {(-1)}^{p(u)}[W^{(1)}_{u,i}t^{-s-1},W^{(2)}_{j,j}t] W^{(1)}_{i,u}t^{s+1}.\label{bru}
\end{align}
By Corollary~\ref{COR}, the first term of the right hand side of \eqref{bru} is equal to
\begin{align}
&\quad-{(-1)}^{p(i)} \displaystyle\sum_{s \geq 0}  \limits\displaystyle\sum_{u=1}^{i}\limits{(-1)}^{p(u)} W^{(1)}_{u,i}t^{-s} [W^{(1)}_{i,u}t^s,W^{(2)}_{j,j}t]\nonumber\\
&=-\delta_{i,j}{(-1)}^{p(i)} \displaystyle\sum_{s \geq 0}  \limits\displaystyle\sum_{u=1}^{i}\limits{(-1)}^{p(u)} W^{(1)}_{u,i}t^{-s}W^{(2)}_{i,u}t^{s+1}+\delta(i\geq j){(-1)}^{p(i)+p(j)} \displaystyle\sum_{s \geq 0}  \limits W^{(1)}_{j,i}t^{-s}W^{(2)}_{i,j}t^{s+1}\nonumber\\
&\quad-\delta_{i,j}{(-1)}^{p(i)}(l-1)\alpha \displaystyle\sum_{s \geq 0}  \limits\displaystyle\sum_{u=1}^{i}\limits {(-1)}^{p(u)}s W^{(1)}_{u,i}t^{-s}W^{(1)}_{i,u}t^s\nonumber\\
&\quad+{(-1)}^{p(i)}(l-1)(lc-1) \displaystyle\sum_{s \geq 0}  \limits sW^{(1)}_{i,i}t^{-s}W^{(1)}_{j,j}t^s\label{aa1}
\end{align}
since by \eqref{844} $\kappa(e_{i,u},e_{j,j})t^{s-1}$ is equal to zero unless $s=0$. Similarly to \eqref{aa1}, we rewrite the second, third, and 4-th terms of the right hand side of \eqref{bru}. By Corollary~\ref{COR}, the second term of the right hand side of \eqref{bru} is equal to
\begin{align}
&\quad-{(-1)}^{p(i)} \displaystyle\sum_{s \geq 0}  \limits\displaystyle\sum_{u=1}^{i}\limits{(-1)}^{p(u)} [W^{(1)}_{u,i}t^{-s},W^{(2)}_{j,j}t] W^{(1)}_{i,u}t^s\nonumber\\
&=-\delta(i\geq j){(-1)}^{p(i)+p(j)} \displaystyle\sum_{s \geq 0}  \limits W^{(2)}_{j,i}t^{1-s} W^{(1)}_{i,j}t^s+\delta_{i,j}{(-1)}^{p(i)} \displaystyle\sum_{s \geq 0}  \limits\displaystyle\sum_{u=1}^{i}\limits{(-1)}^{p(u)}W^{(2)}_{u,i}t^{1-s}W^{(1)}_{i,u}t^s\nonumber\\
&\quad+\delta(i\geq j){(-1)}^{p(i)+p(j)}(l-1)\alpha \displaystyle\sum_{s \geq 0}  \limits s W^{(1)}_{j,i}t^{-s}W^{(1)}_{i,j}t^s\nonumber\\
&\quad-{(-1)}^{p(i)}(l-1)(lc-1) \displaystyle\sum_{s \geq 0}  \limits sW^{(1)}_{j,j}t^{-s}W^{(1)}_{i,i}t^s.\label{aa2}
\end{align}
By Corollary~\ref{COR}, the third term of the right hand side of \eqref{bru} is equal to
\begin{align}
&\quad-{(-1)}^{p(i)} \displaystyle\sum_{s \geq 0} \limits\displaystyle\sum_{u=i+1}^{m+n}\limits {(-1)}^{p(u)}W^{(1)}_{u,i}t^{-s-1} [W^{(1)}_{i,u}t^{s+1},W^{(2)}_{j,j}t]\nonumber\\
&=-\delta_{i,j}{(-1)}^{p(i)} \displaystyle\sum_{s \geq 0}  \limits\displaystyle\sum_{u=i+1}^{m+n}\limits{(-1)}^{p(u)} W^{(1)}_{u,i}t^{-s-1}W^{(2)}_{i,u}t^{s+2}\nonumber\\
&\quad+\delta(i< j){(-1)}^{p(i)+p(j)} \displaystyle\sum_{s \geq 0}  \limits W^{(1)}_{j,i}t^{-s-1}W^{(2)}_{i,j}t^{s+2}\nonumber\\
&\quad-\delta_{i,j}{(-1)}^{p(i)}(l-1)\alpha \displaystyle\sum_{s \geq 0}  \limits\displaystyle\sum_{u=i+1}^{m+n}\limits(s+1){(-1)}^{p(u)} W^{(1)}_{u,i}t^{-s-1}W^{(1)}_{i,u}t^{s+1}.\label{aa3}
\end{align}
By Corollary~\ref{COR}, the 4-th term of the right hand side of \eqref{bru} is equal to
\begin{align}
&\quad-{(-1)}^{p(i)} \displaystyle\sum_{s \geq 0} \limits\displaystyle\sum_{u=i+1}^{m+n}\limits {(-1)}^{p(u)}[W^{(1)}_{u,i}t^{-s-1},W^{(2)}_{j,j}t] W^{(1)}_{i,u}t^{s+1} \nonumber\\
&=-\delta(i< j){(-1)}^{p(i)+p(j)} \displaystyle\sum_{s \geq 0}  \limits W^{(2)}_{j,i}t^{-s} W^{(1)}_{i,j}t^{s+1}\nonumber\\
&\quad+\delta_{i,j}{(-1)}^{p(i)} \displaystyle\sum_{s \geq 0}  \limits\displaystyle\sum_{u=i+1}^{m+n}\limits{(-1)}^{p(u)}W^{(2)}_{u,i}t^{-s}W^{(1)}_{i,u}t^{s+1}\nonumber\\
&\quad+\delta(i<j){(-1)}^{p(i)+p(j)}(l-1)\alpha \displaystyle\sum_{s \geq 0}  \limits (s+1) W^{(1)}_{j,i}t^{-s-1}W^{(1)}_{i,j}t^{s+1}.\label{aa4}
\end{align}
We prepare some notations. We denote the $i$-th term of the right hand side of \eqref{aa1} (resp.\ \eqref{aa2}, \eqref{aa3}, \eqref{aa4}) by $\eqref{aa1}_i$ (resp.\ $\eqref{aa2}_i$, $\eqref{aa3}_i$, $\eqref{aa4}_i$). Let us set
\begin{align*}
A_{i,j}&={(-1)}^{p(j)}\eqref{aa1}_1+{(-1)}^{p(j)}\eqref{aa2}_2+{(-1)}^{p(j)}\eqref{aa3}_1+{(-1)}^{p(j)}\eqref{aa4}_2\\
&=-\delta_{i,j}\displaystyle\sum_{s \geq 0}  \limits\displaystyle\sum_{u=1}^{i}\limits{(-1)}^{p(u)} W^{(1)}_{u,i}t^{-s}W^{(2)}_{i,u}t^{s+1}+\delta_{i,j}\displaystyle\sum_{s \geq 0}  \limits\displaystyle\sum_{u=1}^{i}\limits{(-1)}^{p(u)}W^{(2)}_{u,i}t^{1-s}W^{(1)}_{i,u}t^s\\
&\quad-\delta_{i,j}\displaystyle\sum_{s \geq 0}  \limits\displaystyle\sum_{u=i+1}^{m+n}\limits{(-1)}^{p(u)} W^{(1)}_{u,i}t^{-s-1}W^{(2)}_{i,u}t^{s+2}+\delta_{i,j}\displaystyle\sum_{s \geq 0}  \limits\displaystyle\sum_{u=i+1}^{m+n}\limits{(-1)}^{p(u)}W^{(2)}_{u,i}t^{-s}W^{(1)}_{i,u}t^{s+1},\\
B_{i,j}&={(-1)}^{p(j)}\eqref{aa1}_2+{(-1)}^{p(j)}\eqref{aa2}_1+{(-1)}^{p(j)}\eqref{aa3}_3+{(-1)}^{p(j)}\eqref{aa4}_1\\
&=\delta(i\geq j){(-1)}^{p(i)} \displaystyle\sum_{s \geq 0}  \limits W^{(1)}_{j,i}t^{-s}W^{(2)}_{i,j}t^{s+1}-\delta(i\geq j){(-1)}^{p(i)} \displaystyle\sum_{s \geq 0}  \limits W^{(2)}_{j,i}t^{1-s} W^{(1)}_{i,j}t^s\\
&\quad+\delta(i< j){(-1)}^{p(i)} \displaystyle\sum_{s \geq 0}  \limits W^{(1)}_{j,i}t^{-s-1}W^{(2)}_{i,j}t^{s+2}-\delta(i< j){(-1)}^{p(i)} \displaystyle\sum_{s \geq 0}  \limits W^{(2)}_{j,i}t^{-s} W^{(1)}_{i,j}t^{s+1},\\
C_{i,j}&={(-1)}^{p(j)}\eqref{aa2}_3+{(-1)}^{p(j)}\eqref{aa4}_3\\
&=\delta(i\geq j){(-1)}^{p(i)}(l-1)\alpha \displaystyle\sum_{s \geq 0}  \limits s W^{(1)}_{j,i}t^{-s}W^{(1)}_{i,j}t^s\\
&\quad+\delta(i< j){(-1)}^{p(i)}(l-1)\alpha \displaystyle\sum_{s \geq 0}  \limits  (s+1) W^{(1)}_{j,i}t^{-s-1}W^{(1)}_{i,j}t^{s+1}\\
&={(-1)}^{p(i)}(l-1)\alpha \displaystyle\sum_{s \geq 0}  \limits s W^{(1)}_{j,i}t^{-s}W^{(1)}_{i,j}t^s,\\
D_{i,j}&={(-1)}^{p(j)}\eqref{aa1}_3+{(-1)}^{p(j)}\eqref{aa3}_3\\
&=-\delta_{i,j} (l-1)\alpha\displaystyle\sum_{s \geq 0}  \limits\displaystyle\sum_{u=1}^{i}\limits s{(-1)}^{p(u)} W^{(1)}_{u,i}t^{-s}W^{(1)}_{i,u}t^s\\
&\quad-\delta_{i,j}(l-1)\alpha\displaystyle\sum_{s \geq 0}  \limits\displaystyle\sum_{u=i+1}^{m+n}\limits(s+1){(-1)}^{p(u)} W^{(1)}_{u,i}t^{-s-1}W^{(1)}_{i,u}t^{s+1},\\
\tilde{E}_{i,j}&={(-1)}^{p(j)}\eqref{aa1}_4+{(-1)}^{p(j)}\eqref{aa2}_4\\
&={(-1)}^{p(i)+p(j)} (l-1)(lc-1)\displaystyle\sum_{s \geq 0}  \limits sW^{(1)}_{i,i}t^{-s}W^{(1)}_{j,j}t^s\\
&\quad-{(-1)}^{p(i)+p(j)}(l-1)(lc-1)\displaystyle\sum_{s \geq 0}  \limits sW^{(1)}_{j,j}t^{-s}W^{(1)}_{i,i}t^s.
\end{align*}
Then, we can rewrite $[X_i,{(-1)}^{p(j)}W^{(2)}_{j,j}t]$ as $A_{i,j}+B_{i,j}+C_{i,j}+D_{i,j}+\tilde{E}_{i,j}$.
By exchanging $i$ and $j$, we find that $[X_j,{(-1)}^{p(i)}W^{(2)}_{i,i}t]$ is equal to $A_{j,i}+B_{j,i}+C_{j,i}+D_{j,i}+\tilde{E}_{j,i}.$
We find that the left hand side of \eqref{equat} is equal to
\begin{align*}
\eqref{equat1}+A_{i,j}+B_{i,j}+C_{i,j}+D_{i,j}+\tilde{E}_{i,j}-(A_{j,i}+B_{j,i}+C_{j,i}+D_{j,i}+\tilde{E}_{j,i})+\eqref{equat2}.
\end{align*}
By the definition of $A_{i,j}$ and $D_{i,j}$, we have
\begin{equation}\label{ga0}
A_{i,j}-A_{j,i}=0,\qquad D_{i,j}-D_{j,i}=0.
\end{equation}
By direct computation, we obtain
\begin{gather}
C_{i,j}-C_{j,i}+\eqref{equat1}_7+\eqref{equat1}_8=0,\label{ga1}
\end{gather}
where we denote the $i$-th term of \eqref{equat1} by $\eqref{equat1}_i$. Hence, by \eqref{ga0}, it is enough to obtain the following two relations
\begin{gather}
B_{i,j}-B_{j,i}+\eqref{equat1}_1+\eqref{equat1}_2+\eqref{equat1}_3+\eqref{equat1}_4+\eqref{equat1}_5+\eqref{equat1}_6=0,\label{ga2}\\
\tilde{E}_{i,j}-\tilde{E}_{j,i}+\eqref{equat1}_9+\eqref{equat2}=0.\label{ga3}
\end{gather}

First, we show that \eqref{ga2} holds. Let us compute $B_{i,j}-B_{j,i}$. When $i=j$, it is equal to zero and \eqref{ga2} holds. Suppose that $i>j$. 
Then, we can rewrite $B_{i,j}-B_{j,i}$ as
\begin{align}
&{(-1)}^{p(i)} \displaystyle\sum_{s \geq 0}  \limits W^{(1)}_{j,i}t^{-s}W^{(2)}_{i,j}t^{s+1}-{(-1)}^{p(i)} \displaystyle\sum_{s \geq 0}  \limits W^{(2)}_{j,i}t^{1-s} W^{(1)}_{i,j}t^s\nonumber\\
&\quad-{(-1)}^{p(j)} \displaystyle\sum_{s \geq 0}  \limits W^{(1)}_{i,j}t^{-s-1}W^{(2)}_{j,i}t^{s+2}+{(-1)}^{p(j)} \displaystyle\sum_{s \geq 0}  \limits W^{(2)}_{i,j}t^{-s} W^{(1)}_{j,i}t^{s+1}.\label{bad1}
\end{align}
By Corollary~\ref{COR}, we obtain
\begin{align}
&\quad{(-1)}^{p(i)} \displaystyle\sum_{s \geq 0}  \limits W^{(1)}_{j,i}t^{-s}W^{(2)}_{i,j}t^{s+1}+{(-1)}^{p(j)} \displaystyle\sum_{s \geq 0}  \limits W^{(2)}_{i,j}t^{-s} W^{(1)}_{j,i}t^{s+1}\nonumber\\
&={(-1)}^{p(i)} \displaystyle\sum_{s \geq 0}  \limits W^{(1)}_{j,i}t^{-s-1}W^{(2)}_{i,j}t^{s+2}+{(-1)}^{p(j)} \displaystyle\sum_{s \geq 0}  \limits W^{(2)}_{i,j}t^{1-s} W^{(1)}_{j,i}t^{s}\nonumber\\
&\quad+{(-1)}^{p(i)}W^{(2)}_{i,i}t-{(-1)}^{p(j)}W^{(2)}_{j,j}t\label{bad2}
\end{align}
Appling \eqref{bad2} to \eqref{bad1}, we obtain
\begin{align*}
&\quad B_{i,j}-B_{j,i}\\
&={(-1)}^{p(i)} \displaystyle\sum_{s \geq 0}  \limits W^{(1)}_{j,i}t^{-s-1}W^{(2)}_{i,j}t^{s+2}-{(-1)}^{p(i)} \displaystyle\sum_{s \geq 0}  \limits W^{(2)}_{j,i}t^{1-s} W^{(1)}_{i,j}t^s\nonumber\\
&\quad-{(-1)}^{p(j)} \displaystyle\sum_{s \geq 0}  \limits W^{(1)}_{i,j}t^{-s-1}W^{(2)}_{j,i}t^{s+2}+{(-1)}^{p(j)} \displaystyle\sum_{s \geq 0}  \limits W^{(2)}_{i,j}t^{1-s} W^{(1)}_{j,i}t^{s}\\
&\quad+{(-1)}^{p(i)}W^{(2)}_{i,i}t-{(-1)}^{p(j)}W^{(2)}_{j,j}t.
\end{align*}
We have shown that \eqref{ga2} holds.

Finally, let us compute the left hand side of \eqref{ga3}. 
By direct computation, we obtain
\begin{align*}
&\quad\tilde{E}_{i,j}-\tilde{E}_{j,i}\\
&=2{(-1)}^{p(i)+p(j)}(l-1)(lc-1) \displaystyle\sum_{s \geq 0}  \limits sW^{(1)}_{i,i}t^{-s}W^{(1)}_{j,j}t^s\\
&\quad-2{(-1)}^{p(i)+p(j)}(l-1)(lc-1)\displaystyle\sum_{s \geq 0}  \limits sW^{(1)}_{j,j}t^{-s}W^{(1)}_{i,i}t^s.
\end{align*}
It follows that the left hand side of \eqref{ga3} is equal to
\begin{align*}
&2{(-1)}^{p(i)+p(j)}(l-1)(lc-1)\displaystyle\sum_{s \geq 0}  \limits sW^{(1)}_{i,i}t^{-s}W^{(1)}_{j,j}t^s\\
&\quad-2{(-1)}^{p(i)+p(j)}(l-1)(lc-1)\displaystyle\sum_{s \geq 0}  \limits sW^{(1)}_{j,j}t^{-s}W^{(1)}_{i,i}t^s\\
&\quad-{(-1)}^{p(i)+p(j)}\{(l-1)^2c-(l-1)\}\sum_{s \geq 0} (-sW^{(1)}_{j,j}t^{-s}W^{(1)}_{i,i}t^s+sW^{(1)}_{i,i}t^{-s}W^{(1)}_{j,j}t^s)\\
&\quad-{(-1)}^{p(i)+p(j)}l(lc-1)\sum_{s\geq0}s\{W^{(1)}_{i,i}t^{-s}W^{(1)}_{j,j}t^{s}-W^{(1)}_{j,j}t^{-s}W^{(1)}_{i,i}t^{s}\}\\
&\quad-{(-1)}^{p(i)+p(j)}\sum_{s\geq0}s\{W^{(1)}_{i,i}t^{-s}W^{(1)}_{j,j}t^{s}-W^{(1)}_{j,j}t^{-s}W^{(1)}_{i,i}t^{s}\}\\
&=-{(-1)}^{p(i)+p(j)}c\sum_{s \geq 0} (-sW^{(1)}_{j,j}t^{-s}W^{(1)}_{i,i}t^s+sW^{(1)}_{i,i}t^{-s}W^{(1)}_{j,j}t^s).
\end{align*}
Since $c=0$, this is equal to zero. Thus, \eqref{ga3} holds. We have shown that $[\Phi(H_{i,1}),\Phi(H_{j,1})]=0$.
\end{proof}
Since we have proved Claim~\ref{Claim1234} and Claim~\ref{Claim1236}, we have proven that $\Phi$ is compatible with the defining relations of the affine super Yangian.
\end{proof}
Next, let us show that $\Phi$ is essentially surjective when $\alpha\neq0$.
\begin{Theorem}\label{Main2}
The image of $\Phi$ is dense in $\mathcal{U}(\mathcal{W}^{k}(\mathfrak{gl}(ml|nl),(l^{(m|n)})))$ provided that $\alpha$ is nonzero.
\end{Theorem}
\begin{proof}
Suppose that $\alpha\neq0$. By Theorem~\ref{Tinf}, it is enough to show that the completion of the image of $\Phi$ contains $W^{(1)}_{i,j}t^s$ and $W^{(2)}_{i,j}t^s$ for all $1\leq i,j\leq m+n$ and $s\in\mathbb{Z}$. 

First, we show that $W^{(1)}_{j,j}t^s$ is contained in the completion of the image of $\Phi$. By the definition of $\Phi(H_{i,0})$ and $\Phi(X^\pm_{i,0})$, the image of $\Phi$ contains $({(-1)}^{p(i)}W^{(1)}_{i,i}-{(-1)}^{p(j)}W^{(1)}_{j,j})t$ and $W^{(1)}_{i,j}t^s$ for all $i\neq j$ and $s\in\mathbb{Z}$. Then, by the definition of $\Phi(H_{i,1})$, the completion of the image of $\Phi$ contains
\begin{align*}
&({(-1)}^{p(j)}W^{(2)}_{j,j}-{(-1)}^{p(j+1)}W^{(2)}_{j+1,j+1})t\\
&\quad-\sum_{a\geq0}W^{(1)}_{j,j}t^{-a}W^{(1)}_{j,j}t^a+\sum_{a\geq0}W^{(1)}_{j+1,j+1}t^{-a-1}W^{(1)}_{j+1,j+1}t^{a+1}-{(-1)}^{p(e_{j,j+1})}W^{(1)}_{j,j}W^{(1)}_{j+1,j+1}.
\end{align*}
We take $1\leq r,q\leq m+n$ such that $q\neq r,r+1$. Then, by Corollary~\ref{COR}, we have
\begin{align*}
&\quad[W^{(1)}_{q,r}t^{s-1},({(-1)}^{p(r)}W^{(2)}_{r,r}-{(-1)}^{p(r+1)}W^{(2)}_{r+1,r+1})t-\sum_{a\geq0}W^{(1)}_{r,r}t^{-a}W^{(1)}_{r,r}t^a\\
&\qquad\qquad\qquad\qquad\qquad+\sum_{a\geq0}W^{(1)}_{r+1,r+1}t^{-a-1}W^{(1)}_{r+1,r+1}t^{a+1}-{(-1)}^{p(e_{r,r+1})}W^{(1)}_{r,r}W^{(1)}_{r+1,r+1}]\\
&=-{(-1)}^{p(r)}W^{(2)}_{q,r}t^s+\sum_{a\geq0}W^{(1)}_{q,r}t^{-a+s-1}W^{(1)}_{r,r}t^a\\
&\qquad\qquad+\sum_{a\geq0}W^{(1)}_{r,r}t^{-a-1}W^{(1)}_{q,r}t^{a+s}-{(-1)}^{p(e_{r,r+1})}W^{(1)}_{q,r}t^{s-1}W^{(1)}_{r+1,r+1}.
\end{align*}
Let us set $\sum_{a\geq0}\limits W^{(1)}_{q,r}t^{-a+s-1}W^{(1)}_{r,r}t^a+\sum_{a\geq0}\limits W^{(1)}_{r,r}t^{-a-1}W^{(1)}_{q,r}t^{a+s}-{(-1)}^{p(e_{r,r+1})}W^{(1)}_{q,r}t^{s-1}W^{(1)}_{r+1,r+1}$ as $P_{q,r}^s$. By Lemma~\ref{Lem1}, we obtain
\begin{align*}
&\quad[W^{(1)}_{r,q},P_{q,r}^s]-[W^{(1)}_{r,q}t,P_{q,r}^{s-1}]\\
&=-{(-1)}^{p(q)}l\alpha W^{(1)}_{r,r}t^{s-1}+\delta_{s,1}{(-1)}^{p(e_{r,r+1})+p(q)}l\alpha W^{(1)}_{r+1,r+1}\\
&\qquad\qquad\qquad\qquad+{(-1)}^{p(e_{q,r})}W^{(1)}_{q,r}t^{s-1}W^{(1)}_{r,q}-W^{(1)}_{r,q}W^{(1)}_{q,r}t^{s-1}.
\end{align*}
Then, by Lemma~\ref{Lem1} and Corollary~\ref{COR}, we have
\begin{align}
&\quad[W^{(1)}_{r,q},W^{(2)}_{q,r}t^s-{(-1)}^{p(r)}P_{q,r}^s]-[W^{(1)}_{r,q}t,W^{(2)}_{q,r}t^{s-1}-{(-1)}^{p(r)}P_{q,r}^{s-1}]\nonumber\\
&=-(l-1)\alpha W^{(1)}_{q,q}t^{s-1}+{(-1)}^{p(e_{q,r})}l\alpha W^{(1)}_{r,r}t^{s-1}-\delta_{s,1}{(-1)}^{p(e_{q,r+1})}l\alpha W^{(1)}_{r+1,r+1}\nonumber\\
&\quad-{(-1)}^{p(q)}W^{(1)}_{q,r}t^{s-1}W^{(1)}_{r,q}+{(-1)}^{p(r)}W^{(1)}_{r,q}W^{(1)}_{q,r}t^{s-1}\nonumber\\
&=\alpha W^{(1)}_{q,q}t^{s-1}-l\alpha (W^{(1)}_{q,q}t^{s-1}-{(-1)}^{p(e_{q,r})}W^{(1)}_{r,r}t^{s-1})-\delta_{s,1}{(-1)}^{p(e_{q,r+1})}l\alpha W^{(1)}_{r+1,r+1}\nonumber\\
&\quad-{(-1)}^{p(q)}W^{(1)}_{q,r}t^{s-1}W^{(1)}_{r,q}+{(-1)}^{p(r)}W^{(1)}_{r,q}W^{(1)}_{q,r}t^{s-1}.\label{equat432}
\end{align}
We find that $\alpha W^{(1)}_{q,q}t^{s-1}-\delta_{s,1}{(-1)}^{p(e_{q,r+1})}l\alpha W^{(1)}_{r+1,r+1}$
is contained in the completion of the image of $\Phi$ by \eqref{equat432},
we have shown that the completion of the image of $\Phi$ contains $W^{(1)}_{q,q}t^s$. 

Since we have already shown that the completion of the image of $\Phi$ contains $\{W^{(1)}_{i,j}t^s\mid1\leq i,j\leq m+n\}$, we find that the completion of the image of $\Phi$ contains $({(-1)}^{p(i)}W^{(2)}_{i,i}-{(-1)}^{p(j)}W^{(2)}_{j,j})t$ and $W^{(2)}_{i,i\pm1}t$ for all $1\leq i,j\leq m+n$ by the definition of $\Phi(H_{i,1})$ and $\Phi(X^\pm_{i,1})$. 
Since there exists a pair $(i,j)$ such that $p(i)=p(j)$, it is enough to prove that $W^{(2)}_{i,j}t^s\ (i\neq j)$, $({(-1)}^{p(i)}W^{(2)}_{i,i}-{(-1)}^{p(j)}W^{(2)}_{j,j})t^s$, $W^{(1)}_{j,j}t^s$, and $W^{(2)}_{i,i}t^s+W^{(2)}_{j,j}t^s$ are contained in the image of the completion of $\Phi$.
Next, we show that the completion of the image of $\Phi$ contains $W^{(2)}_{i,j}t^s\ (i\neq j)$. By Corollary~\ref{COR}, we have
\begin{gather*}
[W^{(1)}_{i,j}t^{s-1},({(-1)}^{p(j)}W^{(2)}_{j,j}-{(-1)}^{p(j+1)}W^{(2)}_{j+1,j+1})t]=-{(-1)}^{p(j)}W^{(2)}_{i,j}t^s\ (\text{if }i\neq j, j+1),\\
[W^{(1)}_{j+1,j}t^{s-1},({(-1)}^{p(j)}W^{(2)}_{j,j}-{(-1)}^{p(j-1)}W^{(2)}_{j-1,j-1})t]=-{(-1)}^{p(j)}W^{(2)}_{j+1,j}t^s.
\end{gather*}
Thus, $W^{(2)}_{i,j}t^s\ (i\neq j)$ is contained in the completion of the image of $\Phi$. By \eqref{241}, we obtain
\begin{align*}
&[({(-1)}^{p(i)}W^{(2)}_{i,i}-{(-1)}^{p(j)}W^{(2)}_{j,j})t,({(-1)}^{p(i)}W^{(2)}_{i,i}-{(-1)}^{p(j)}W^{(2)}_{j,j})t^s]\\
&\qquad-[({(-1)}^{p(i)}W^{(2)}_{i,i}-{(-1)}^{p(j)}W^{(2)}_{j,j}),({(-1)}^{p(i)}W^{(2)}_{i,i}-{(-1)}^{p(j)}W^{(2)}_{j,j})t^{s+1}]\\
&=({(-1)}^{p(i)}W^{(2)}_{i,i}-{(-1)}^{p(j)}W^{(2)}_{j,j})_{(0)}({(-1)}^{p(i)}W^{(2)}_{i,i}-{(-1)}^{p(j)}W^{(2)}_{j,j})t^{s+1}\\
&\quad+({(-1)}^{p(i)}W^{(2)}_{i,i}-{(-1)}^{p(j)}W^{(2)}_{j,j})_{(1)}({(-1)}^{p(i)}W^{(2)}_{i,i}-{(-1)}^{p(j)}W^{(2)}_{j,j})t^s\\
&\quad-({(-1)}^{p(i)}W^{(2)}_{i,i}-{(-1)}^{p(j)}W^{(2)}_{j,j})_{(0)}({(-1)}^{p(i)}W^{(2)}_{i,i}-{(-1)}^{p(j)}W^{(2)}_{j,j})t^{s+1}\\
&=({(-1)}^{p(i)}W^{(2)}_{i,i}-{(-1)}^{p(j)}W^{(2)}_{j,j})_{(1)}({(-1)}^{p(i)}W^{(2)}_{i,i}-{(-1)}^{p(j)}W^{(2)}_{j,j})t^s.
\end{align*}
By Lemma~\ref{Lem4}, provided that $i\neq j$, it is equal to
\begin{align*}
&\quad-2\alpha(W^{(2)}_{i,i}+W^{(2)}_{j,j})t^s-2{(-1)}^{p(i)}W^{(2)}_{i,i}-2{(-1)}^{p(j)}W^{(2)}_{j,j}+2({(-1)}^{p(i)}W^{(2)}_{i,i}+{(-1)}^{p(j)}W^{(2)}_{j,j})t^s\\
&\qquad+(\text{the terms consisting of $\{W^{(1)}_{i,j}\ (1\leq i,j\leq m+n),\ W^{(2)}_{i,j}\ (i\neq j)\}$})\\
&=-2\alpha(W^{(2)}_{i,i}+W^{(2)}_{j,j})t^s\\
&\quad+(\text{the terms consisting of $\{W^{(1)}_{i,j}\ (1\leq i,j\leq m+n),\ W^{(2)}_{i,j}\ (i\neq j)\}$}).
\end{align*}
Thus, the completion of the image of $\Phi$ contains $W^{(2)}_{i,i}t^s+W^{(2)}_{j,j}t^s$. 
\end{proof}
We obtain the following theorem in the similar proof as that of Theorem~\ref{Main} and Theorem~\ref{Main2}.
\begin{Theorem}\label{T198}
We assume that $m\geq3$ and $l\geq2$. Let us set 
\begin{gather*}
\ve_1=\dfrac{k+(l-1)m}{m},\quad\ve_2=-1-\dfrac{k+(l-1)m}{m}.
\end{gather*}
Then, there exists an algebra homomorphism 
\begin{equation*}
\Phi\colon Y_{\ve_1,\ve_2}(\widehat{\mathfrak{sl}}(m))\to \mathcal{U}(\mathcal{W}^{k}(\mathfrak{gl}(ml),(l^{m})))
\end{equation*} 
determined by the same formula as that of Theorem~\ref{Main} under the assumption that $n=0$. Moreover, the image of $\Phi$ is dense in $\mathcal{U}(\mathcal{W}^{k}(\mathfrak{gl}(ml),(l^{m})))$ provided that $k+(l-1)m\neq0$.
\end{Theorem}

\section{Stukopin's Yangians and rectangular \\\qquad\qquad\qquad finite $W$-superalgebras of type $A$}

In this section, we explain how the homomorphism $\Phi$ deduce the map from the Stukopin's Yangian to rectangular finite $W$-superalgebras of type $A$.

First, let us recall the definition of the Zhu algebra (\cite{Zhu}).
Let $V$ be a vertex algebra and $H\in\End(V)$ be the Hamiltonian of $V$. We can define a degree on $\mathcal{U}(V)$ by
\begin{equation*}
\text{deg}(v\otimes t^s)=r+s+1\text{ if }H(v)=-rv.
\end{equation*}
We denote the set of degree $r$ elements of $\mathcal{U}(V)$ by $\mathcal{U}(V)_r$.
In the case that $V$ is a rectangular $W$-superalgebra $W^k(\mathfrak{gl}(ml|nl),(l^{(m|n)}))$, we obtain
$\text{deg}(W^{(r)}_{i,j}t^s)=s-r+1$.

By Theorem~A.2.11 in \cite{NT}, the Zhu algebra of $V$ can be defined as an associative algebra 
\begin{equation*}
Zh(V)=\mathcal{U}(V)_0/\sum_{r>0}\limits\mathcal{U}(V)_{-r}\mathcal{U}(V)_r.
\end{equation*}
The finite $W$-superalgebra is defined by \cite{Pr}. A finite $W$-superalgebra $W^{\text{fin}}(\mathfrak{g},f)$ is an associative algebra associated with a finite dimensional reductive Lie superalgebra $\mathfrak{g}$ and its even nilpotent element $f$.
By \cite{A1}, \cite{FKW}, and \cite{DSK1} a finite $W$-superalgebra $W^{\text{fin}}(\mathfrak{g},f)$ is the Zhu algebra of the $W^{k}(\mathfrak{g},f)$. In the case when $\mathfrak{g}=\mathfrak{gl}(ml|nl)$ and $f$ is a nilpotent element whose Jordan block is of type $(l^{(m|n)})$, we call $W^{\text{fin}}(\mathfrak{g},f)$ the rectangular finite $W$-superalgebra and denote it by  $W^{\text{fin}}(\mathfrak{gl}(ml|nl),(l^{(m|n)}))$.

Peng \cite{Pe} constructed a surjective homomorphism from shifted super Yangians to finite $W$-superalgebras of type $A$.
Especially, in the rectangular case, he gave a surjective homomorphism from the Nazarov's Yangian (\cite{Na}) to the rectangular finite $W$-superalgebras.
The definition of the Nazarov's Yangian is as follows.
\begin{Definition}
The Nazarov's Yangian $Y(\mathfrak{gl}(m|n))$ is an associative superalgebra whose generators are $\{t^{(r)}_{i,j}\mid r\geq0, 1\leq i,j\leq m+n\}$ and defining relations are
\begin{gather*}
t^{(0)}_{i,j}=\delta_{i,j},\\
[t^{(r+1)}_{i,j},t^{(s)}_{u,v}]-[t^{(r)}_{i,j},t^{(s+1)}_{u,v}]=(-1)^{p(i)p(j)+p(i)p(u)+p(j)p(u)}(t^{(r)}_{u,j}t^{(s)}_{i,v}-t^{(s)}_{u,j}t^{(r)}_{i,v}),
\end{gather*}
where $t^{(r)}_{i,j}$ is $\begin{cases}
\text{even}&\text{if }p(i)+p(j)=0,\\
\text{odd}&\text{if }p(i)+p(j)=1.
\end{cases}$
\end{Definition}
Let $\tilde{p}$ be a natural projection from $\mathcal{U}(W^k(\mathfrak{gl}(ml|nl),(l^{(m|n)})))_0$ to $\text{Zhu}(W^k(\mathfrak{gl}(ml|nl),(l^{(m|n)})))$. Then, $\{\tilde{p}(W^{(r)}_{i,j}t^{r-1})\mid1\leq i,j\leq m+n,1\leq r\leq l\}$ becomes the strong generators of the finite $W$-superalgebra $W^{\text{fin}}(\mathfrak{gl}(ml|nl),(l^{(m|n)}))$.
The homomorphism given by Peng is written down as follows;
\begin{equation*}
\widehat{\Phi}\colon Y(\mathfrak{gl}(m|n))\to W^{\text{fin}}(\mathfrak{gl}(ml|nl),(l^{(m|n)})),
\end{equation*}
determined by $t^{(r)}_{i,j}\mapsto \begin{cases}
{(-1)}^{p(i)}\tilde{p}(W^{(r)}_{i,j}t^{r-1})&\text{ if }r\leq l,\\
0&\text{ if }r>l.
\end{cases}$

The Nazarov's Yangian has a subalgebra which is a deformation of the universal enveloping algebra of the current algebra associated with $\mathfrak{sl}(m|n)$. The subalgebra is called the Stukopin's Yangian (\cite{S}). We recall the definition of the Stukopin's Yangian (see \cite{G}).
\begin{Definition}
The Sukopin's Yangian $Y_{\hbar}(\mathfrak{sl}(m|n))$ is the associative superalgebra over $\mathbb{C}$ generated by $x_{i,r}^{+}, x_{i,r}^{-}, h_{i,r}$ $(i \in \{1,\cdots, m+n-1\}, r \in \mathbb{Z}_{\geq 0})$ with two parameters $\hbar\in \mathbb{C}$ subject to the following defining relations:
\begin{gather}
	[h_{i,r}, h_{j,s}] = 0, \\
	[x_{i,r}^{+}, x_{j,s}^{-}] = \delta_{ij} h_{i, r+s}, \\
	[h_{i,0}, x_{j,r}^{\pm}] = \pm a_{ij} x_{j,r}^{\pm},\\
	[h_{i, r+1}, x_{j, s}^{\pm}] - [h_{i, r}, x_{j, s+1}^{\pm}] 
	= \pm a_{ij} \dfrac{\hbar}{2} \{h_{i, r}, x_{j, s}^{\pm}\},\\
	[x_{i, r+1}^{\pm}, x_{j, s}^{\pm}] - [x_{i, r}^{\pm}, x_{j, s+1}^{\pm}] 
	= \pm a_{ij}\dfrac{\hbar}{2} \{x_{i, r}^{\pm}, x_{j, s}^{\pm}\},\\
	\sum_{w \in \mathfrak{S}_{1 + |a_{ij}|}}[x_{i,r_{w(1)}}^{\pm}, [x_{i,r_{w(2)}}^{\pm}, \dots, [x_{i,r_{w(1 + |a_{ij}|)}}^{\pm}, x_{j,s}^{\pm}]\dots]] = 0\  (i \neq j),\\
	[x^\pm_{m,r},x^\pm_{m,s}]=0,\\
	[[x^\pm_{m-1,r},x^\pm_{m,0}],[x^\pm_{m,0},x^\pm_{m+1,s}]]=0.,
\end{gather}
where $x^\pm_{m,r}$ is
\end{Definition}
Similarly to the affine super Yangian,
we note that $Y_{\hbar}(\mathfrak{sl}(m|n))$ is generated by $\{h_{i,0},x^\pm_{i,0},h_{i,1}\}$. For $\hbar\neq0$, the embedding $\iota_\hbar\colon Y_{\hbar}(\mathfrak{sl}(m|n))\to Y(\mathfrak{gl}(m|n))$ is given by 
\begin{gather*}
\iota_\hbar(h_{i,0})=t^{(1)}_{i,i}-t^{(1)}_{i+1,i+1},\quad \iota_\hbar(x^+_{i,0})={(-1)}^{p(i)}t^{(1)}_{i,i+1},\quad \iota_\hbar(x^-_{i,0})= {(-1)}^{p(E_{i,i+1})}t^{(1)}_{i+1,i},\\
\begin{align*}
\iota_\hbar(h_{i,1})&=
-\hbar t^{(2)}_{i,i}+\hbar t^{(2)}_{i+1,i+1}\\
&\quad-\dfrac{(i-2\delta(i\geq m+1)(i-m))}{2}\hbar(t^{(1)}_{i,i}-t^{(1)}_{i+1,i+1}) -\hbar t^{(1)}_{i,i}t^{(1)}_{i+1,i+1} \\
&\quad+ \hbar\displaystyle\sum_{u=1}^{i}\limits t^{(1)}_{i,u}t^{(1)}_{u,i}-\hbar\displaystyle\sum_{u=1}^{i}\limits t^{(1)}_{i+1,u}t^{(1)}_{u,i+1}.
\end{align*}
\end{gather*}

Setting $\ve_1=\ve_2$, there exists a natural homomorphism $\omega$ from the Stukopin's Yangian to the affine super Yangian determined by
\begin{gather*}
\omega(h_{i,0})=H_{i,0},\qquad\omega(x^\pm_{i,0})=X^\pm_{i,0}.
\end{gather*}
Assuming that $\ve_1=\ve_2=-\dfrac{1}{2}$,
by the relation $\text{deg}(W^{(r)}_{i,j}t^s)=s-r+1$, we find that the image of $\Phi\circ\omega$ is contained in $\mathcal{U}(W^k(\mathfrak{gl}(ml|nl),(l^{(m|n)})))_0$. Then, we obtain a homomorphism
\begin{equation*}
\tilde{\Phi}=p\circ\Phi\circ\omega\colon Y_{-1}(\mathfrak{sl}(m|n))\to\text{Zh}(W^k(\mathfrak{gl}(ml|nl),(l^{(m|n)}))).
\end{equation*}
By $\text{deg}(W^{(r)}_{i,j}t^s)=s-r+1$, we can explicitly write down $\tilde{\Phi}$ as follows;
\begin{gather*}
\tilde{\Phi}(H_{i,0})=
{(-1)}^{p(i)}\tilde{p}(W^{(1)}_{i,i})-{(-1)}^{p(i+1)}\tilde{p}(W^{(1)}_{i+1,i+1}),
\\
\tilde{\Phi}(X^+_{i,0})=
\tilde{p}(W^{(1)}_{i+1,i}),
\qquad \tilde{\Phi}(X^-_{i,0})=
{(-1)}^{p(i)}\tilde{p}(W^{(1)}_{i,i+1}),
\end{gather*}
\begin{align*}
\tilde{\Phi}(H_{i,1})&=
{(-1)}^{p(i)}\tilde{p}(W^{(2)}_{i,i}t)-{(-1)}^{p(i+1)}\tilde{p}(W^{(2)}_{i+1,i+1}t)\\
&\quad+\dfrac{i-2\delta(i\geq m+1)(i-m)}{2}({(-1)}^{p(i)}\tilde{p}(W^{(1)}_{i,i})-{(-1)}^{p(i+1)}\tilde{p}(W^{(1)}_{i+1,i+1}))\\
&\quad+{(-1)}^{p(E_{i,i+1})}  \tilde{p}(W^{(1)}_{i,i})\tilde{p}(W^{(1)}_{i+1,i+1})\\
&\quad-{(-1)}^{p(i)}\displaystyle\sum_{u=1}^{i}\limits{(-1)}^{p(u)} \tilde{p}(W^{(1)}_{u,i})\tilde{p}(W^{(1)}_{i,u})+{(-1)}^{p(i+1)}\displaystyle\sum_{u=1}^{i}\limits{(-1)}^{p(u)}\tilde{p}(W^{(1)}_{u,i+1})\tilde{p}(W^{(1)}_{i+1,u}).
\end{align*}
\begin{Theorem}\label{Finite}
We obtain the commutativity $\tilde{\Phi}=\widehat{\Phi}\circ\iota_{-1}$.
\end{Theorem}
\begin{proof}
It is enough to show that
\begin{equation*}
\tilde{\Phi}(x)=\widehat{\Phi}\circ\iota_{-1}(x)\text{ for }x=h_{i,1}, x^\pm_{i,0},h_{i,0}.
\end{equation*}
We only show the case when $x=h_{i,1}$ since other cases are trivial.
By a direct computation, we obtain
\begin{align*}
&\quad\widehat{\Phi}\circ\iota_{-1}(h_{i,1})\\
&=\widehat{\Phi}(t^{(2)}_{i,i})-\widehat{\Phi}(t^{(2)}_{i+1,i+1})\\
&\quad+\dfrac{(i-2\delta(i\geq m+1)(i-m))}{2}(\widehat{\Phi}(t^{(1)}_{i,i})-\widehat{\Phi}(t^{(1)}_{i+1,i+1}))+ \widehat{\Phi}(t^{(1)}_{i,i})\widehat{\Phi}(t^{(1)}_{i+1,i+1})\\
&\quad-\displaystyle\sum_{k=1}^{i}\limits \widehat{\Phi}(t^{(1)}_{i,k})\widehat{\Phi}(t^{(1)}_{k,i})+\displaystyle\sum_{k=1}^{i}\limits\widehat{\Phi}(t^{(1)}_{i+1,k})\widehat{\Phi}(t^{(1)}_{k,i+1})\\
&={(-1)}^{p(i)}\tilde{p}(W^{(2)}_{i,i}t)-{(-1)}^{p(i+1)}\tilde{p}(W^{(2)}_{i+1,i+1}t)\\
&\quad+\dfrac{(i-2\delta(i\geq m+1)(i-m))}{2}({(-1)}^{p(i)} \tilde{p}(W^{(1)}_{i,i})-{(-1)}^{p(i+1)}\tilde{p}(W^{(1)}_{i+1,i+1}))\\
&\quad +{(-1)}^{p(E_{i,i+1})} \tilde{p}(W^{(1)}_{i,i})\tilde{p}(W^{(1)}_{i+1,i+1})\\
&\quad-{(-1)}^{p(i)}\displaystyle\sum_{k=1}^{i}\limits{(-1)}^{p(k)} \tilde{p}(W^{(1)}_{i,k})\tilde{p}(W^{(1)}_{k,i})+{(-1)}^{p(i+1)}\displaystyle\sum_{k=1}^{i}\limits{(-1)}^{p(k)}\tilde{p}(W^{(1)}_{i+1,k})\tilde{p}(W^{(1)}_{k,i+1}).
\end{align*}
This is nothing but $\tilde{\Phi}(h_{i,1})$.
\end{proof}
\begin{Remark}
For $\hbar\neq 0$, $Y_{\hbar}(\mathfrak{sl}(m|n))$ is isomorphic to $Y_{-1}(\mathfrak{sl}(m|n))$. Thus, by Theorem~\ref{Finite}, the homomorphism $\Phi$ deduce the map from $Y_{\hbar}(\mathfrak{sl}(m|n))$ to rectangular finite $W$-superalgebras of type $A$ for all $\hbar\neq 0$.
\end{Remark}

\begin{appendices}
\renewcommand{\thesection}{A\arabic{section}}
\section{The proof of Theorem~\ref{Tinf}}
In this section, we prove Theorem~\ref{Tinf}.
We define a grading on $\mathfrak{b}$ by setting $\text{deg}(x)=j$ if $x\in\mathfrak{b}\cap\mathfrak{g}_{j}$. Since
\begin{equation*}
\{\displaystyle\sum_{s=1}^{l-r}\limits e_{(r+s-1)(m+n)+j,(s-1)(m+n)+i}\mid0\leq r\leq l-1,\ 1\leq i,j\leq m+n\}
\end{equation*}
forms a basis of $\mathfrak{gl}(ml|nl)^f=\{g\in\mathfrak{gl}(ml|nl)|[f,g]=0\}$, it is enough to show that $W^{(1)}_{i,j}$ and $W^{(2)}_{i,j}$ generate the term whose form is
\begin{equation*}
\displaystyle\sum_{s=1}^{l-r}\limits e_{(r+s-1)(m+n)+j,(s-1)(m+n)+i}[-1]+\text{higher terms}
\end{equation*}
for all $0\leq r\leq l-1,\ 1\leq i,j\leq m+n$ by Theorem~4.1 of \cite{KW1}.  
We show that $W^{(1)}_{i,j}$ and $W^{(2)}_{i,j}$ generate these terms by two claims, that is, Claim~\ref{T1} and Claim~\ref{T3}. In Claim~\ref{T1} below, we show that $W^{(1)}_{i,j}$ and $W^{(2)}_{i,j}$ generate the term whose form is
\begin{align*}
&{(-1)}^{p(i)}\sum_{s=1}^{l-r}e_{(r+s-1)(m+n)+i,(s-1)(m+n)+i}[-1]\\
&\qquad\qquad-{(-1)}^{p(i+1)}\sum_{s=1}^{l-r}e_{(r+s-1)(m+n)+i+1,(s-1)(m+n)+i+1}[-1]+\text{higher terms}
\end{align*}
or
\begin{equation*}
\sum_{s=1}^{l-r}e_{(r+s-1)(m+n)+j,(s-1)(m+n)+i}[-1]+\text{higher terms}\ (i\neq j)
\end{equation*}
for all $0\leq r\leq l-1$. In Claim~\ref{T3} below, we prove that $W^{(1)}_{i,j}$ and $W^{(2)}_{i,j}$ generate the term whose form is
\begin{equation*}
\sum_{s=1}^{l-r}e_{(r+s-1)(m+n)+i,(s-1)(m+n)+i}[-1]+\text{higher terms}
\end{equation*}
for all $1\leq r\leq l-1$. 
Since $\displaystyle\sum_{s=1}^{l-0}\limits e_{(0+s-1)(m+n)+i,(s-1)(m+n)+i}[-1]$ is nothing but $W^{(1)}_{i,i}$, Theorem~\ref{Tinf} is derived from Claim~\ref{T1} and Claim~\ref{T3}.

In order to prove Claims \ref{T1} and \ref{T3}, we prepare the following claim.
\begin{Claim}\label{T0}
\textup{(1)}\ The following equation holds for all $0\leq w\leq l-1,\ 1\leq i,j,u,v\leq m+n$;
\begin{align}
&\phantom{{}={}}(\sum_{s=1}^{l-1}e_{s(m+n)+j,(s-1)(m+n)+i}[-1])_{(0)}\sum_{t=1}^{l-w}e_{(w+t-1)(m+n)+u,(t-1)(m+n)+v}[-1]\nonumber\\
&=\delta_{i,u}\sum_{t=1}^{l-w-1}e_{(w+t)(m+n)+j,(t-1)(m+n)+v}[-1]\nonumber\\
&\qquad\qquad-\delta_{j,v}{(-1)}^{p(e_{i,j})p(e_{u,v})}\sum_{t=1}^{l-w-1}e_{(w+t)(m+n)+u,(t-1)(m+n)+i}[-1].\label{nom}
\end{align}
\textup{(2)}\ We obtain
\begin{align}
&\phantom{{}={}}(W^{(1)}_{i,j})_{(0)}(\sum_{s=1}^{l-r}e_{(r+s-1)(m+n)+x,(s-1)(m+n)+y}[-1])\nonumber\\
&=\delta_{i,x}\sum_{s=1}^{l-r}e_{(r+s-1)(m+n)+j,(s-1)(m+n)+y}[-1]\nonumber\\
&\qquad\qquad-\delta_{j,y}{(-1)}^{p(e_{i,j})p(e_{x,y})}\sum_{s=1}^{l-r}e_{(r+s-1)(m+n)+x,(s-1)(m+n)+i}[-1]\label{non}
\end{align}
for all $0\leq r\leq l-1,\ 1\leq i,j,x,y\leq m+n$.
\end{Claim}
Claim~\ref{T0} is proven by direct computation. We omit the proof. By \eqref{nom} and \eqref{non}, it is easy to obtain the following claim.
\begin{Claim}\label{T1}
\textup{(1)}\ For all $0\leq r\leq l-1$, the elements $W^{(1)}_{i,j}$ and $W^{(2)}_{i,j}$ generate the term whose form is
\begin{equation*}
\displaystyle\sum_{s=1}^{l-r}\limits e_{(r+s-1)(m+n)+i,(s-1)(m+n)+j}[-1]+\text{higher terms}\ (i\neq j).
\end{equation*}
\textup{(2)}\ For all $0\leq r\leq l-1$, the elements $W^{(1)}_{i,j}$ and $W^{(2)}_{i,j}$ generate the term whose form is
\begin{align*}
&{(-1)}^{p(i)}\sum_{s=1}^{l-r}e_{(r+s-1)(m+n)+i,(s-1)(m+n)+i}[-1]\\
&\qquad\qquad-{(-1)}^{p(i+1)}\sum_{s=1}^{l-r}e_{(r+s-1)(m+n)+i+1,(s-1)(m+n)+i+1}[-1]+\text{higher terms}.
\end{align*}
\end{Claim}
\begin{proof}
First, let us show (1). Since $W^{(2)}_{i,j}$ has the form such that 
\begin{equation*}
\displaystyle\sum_{s=1}^{l-1}\limits e_{s(m+n)+j,(s-1)(m+n)+i}[-1]+\text{degree $0$ terms},
\end{equation*}
we obtain
\begin{align*}
((W^{(2)}_{i,i})_{(0)})^rW^{(1)}_{j,i}=\left((\displaystyle\sum_{s=1}^{l-1}\limits e_{s(m+n)+i,(s-1)(m+n)+i}[-1])_{(0)}\right)^rW^{(1)}_{j,i}+\text{higher terms}
\end{align*}
for all $i\neq j,\ 0\leq r\leq l-1$.
By \eqref{nom}, we have
\begin{align*}
((W^{(2)}_{i,i})_{(0)})^rW^{(1)}_{j,i}=\displaystyle\sum_{s=1}^{l-r}\limits e_{(r+s-1)(m+n)+i,(s-1)(m+n)+j}[-1]+\text{higher terms}.
\end{align*}
Thus, we have proved (1).

Next, let us prove (2). By (1), the element whose form is 
\begin{equation*}
\displaystyle\sum_{s=1}^{l-r}\limits e_{(r+s-1)(m+n)+i,(s-1)(m+n)+i+1}[-1]+\text{higher terms}
\end{equation*}
is generated by $W^{(1)}_{i,j}$ and $W^{(2)}_{i,j}$.
By \eqref{non}, we have
\begin{align*}
&\quad(W^{(1)}_{i,i+1})_{(0)}\big(\displaystyle\sum_{s=1}^{l-r}\limits e_{(r+s-1)(m+n)+i,(s-1)(m+n)+i+1}[-1]+\text{higher terms}\big)\\
&=\sum_{s=1}^{l-r}e_{(r+s-1)(m+n)+i+1,(s-1)(m+n)+i+1}[-1]\\
&\qquad-{(-1)}^{p(e_{i,i+1})}\sum_{s=1}^{l-r}e_{(r+s-1)(m+n)+i,(s-1)(m+n)+i}[-1]+\text{higher terms}.
\end{align*}
Thus, we have proved (2).
\end{proof}
\begin{Claim}\label{T3}
The elements $W^{(1)}_{i,j}$ and $W^{(2)}_{i,j}$ generate the term whose form is 
\begin{equation*}
\displaystyle\sum_{1\leq t\leq l-r}\limits e_{(t+r-1)(m+n)+i,(t-1)(m+n)+i}[-1]+\text{higher terms}
\end{equation*}
for all $1\leq r\leq l-1$.
\end{Claim}
\begin{proof}
It is enough to show that
\begin{align}
&\quad(W^{(2)}_{i,i})_{(1)}(W^{(1)}_{i,i+1})_{(0)}\{{(W^{(2)}_{i,i})}_{(0)}\}^{r}W^{(1)}_{i+1,i}\nonumber\\
&={(-1)}^{p(e_{i,i+1})}\alpha\displaystyle\sum_{1\leq t\leq l}\limits e_{(t-1)(m+n)+i,(t-r-1)(m+n)+i}[-1]\nonumber\\
&\quad+{(-1)}^{p(i+1)}r\sum_{1\leq t\leq l}e_{(t-1)(m+n)+i,(t-r-1)(m+n)+i}[-1]\nonumber\\
&\quad-{(-1)}^{p(i)}r\sum_{1\leq t\leq l}e_{(t-1)(m+n)+i+1,(t-r-1)(m+n)+i+1}[-1]+\text{higher terms}\label{929}
\end{align}
since we have already shown that 
\begin{align*}
\sum_{1\leq t\leq l}(e_{(t-1)(m+n)+i,(t-r-1)(m+n)+i}[-1]-{(-1)}^{p(e_{i,i+1})} e_{(t-1)(m+n)+i+1,(t-r-1)(m+n)+i+1}[-1])
\end{align*}
is generated by $W^{(1)}_{i,j}$ and $W^{(2)}_{i,j}$. Let us set 
\begin{align*}
Z=\displaystyle\sum_{1\leq s\leq l-1}\limits e_{s(m+n)+i,(s-1)(m+n)+i}[-1],\quad W=W^{(2)}_{i,i}-Z.
\end{align*}
The element $W^{(2)}_{i,i}$ is the sum of degree $-1$ element $Z$ and degree $0$ element $W$.
We can rewrite the left hand side of \eqref{929} as
\begin{align}
&Z_{(1)}(W^{(1)}_{i,i+1})_{(0)}(Z_{(0)})^{r}W^{(1)}_{i+1,i}+W_{(1)}(W^{(1)}_{i,i+1})_{(0)}(Z_{(0)})^rW^{(1)}_{i+1,i}\nonumber\\
&+\displaystyle\sum_{1\leq d\leq r}\limits Z_{(1)}(W^{(1)}_{i,i+1})_{(0)}(Z_{(0)})^{r-d}W_{(0)}(Z_{(0)})^{d-1}W^{(1)}_{i+1,i}+\text{higher terms}.\label{no7}
\end{align}
In order to simplify the notation, here after, we denote $\displaystyle\sum_{a\leq s\leq l-b}\limits e_{(b+s-1)(m+n)+i,(s-a)(m+n)+j}[-u]$ by $\displaystyle\sum_{1\leq s\leq l}\limits e_{(b+s-1)(m+n)+i,(s-a)(m+n)+j}[-u]$.
Let us compute the each terms of \eqref{no7}. First, we compute the first term of \eqref{no7}. By \eqref{nom} and \eqref{non}, we have
\begin{align}
&\quad(W^{(1)}_{i,i+1})_{(0)}(Z_{(0)})^{r}W^{(1)}_{i+1,i}\nonumber\\
&=\displaystyle\sum_{1\leq t\leq l}\limits e_{(t-1)(m+n)+i+1,(t-r-1)(m+n)+i+1}[-1]\nonumber\\
&\qquad\qquad\qquad-{(-1)}^{p(e_{i,i+1})}\displaystyle\sum_{1\leq t\leq l}\limits e_{(t-1)(m+n)+i,(t-r-1)(m+n)+i}[-1].\label{no8}
\end{align}
Applying \eqref{no8} to the first term of \eqref{no7}, we obtain
\begin{align*}
&\quad\text{the first term of \eqref{no7}}\\
&=(\displaystyle\sum_{1\leq t\leq l}\limits e_{s(m+n)+i,(s-1)(m+n)+i}[-1])_{(1)}(\text{the right hand side of \eqref{no8}})\\
&=0
\end{align*}
since $\kappa(e_{s(m+n)+j,(s-1)n+j},e_{(t-1)(m+n)+i,(t-r-1)(m+n)+i})=0$.
Next, let us compute the second term of \eqref{no7}. By \eqref{no8}, it is the sum of
\begin{align}
&{(-1)}^{p(i)}(\displaystyle\sum_{\substack{r_1<r_2\\1\leq u\leq m+n}}\limits e^{(r_1)}_{u,i}[-1]e^{(r_2)}_{i,u}[-1])_{(1)}(\displaystyle\sum_{1\leq t\leq l}\limits e_{(t-1)(m+n)+i+1,(t-r-1)(m+n)+i+1}[-1])\nonumber\\
&\qquad-{(-1)}^{p(i+1)}(\displaystyle\sum_{\substack{r_1<r_2\\1\leq u\leq m+n}}\limits e^{(r_1)}_{u,i}[-1]e^{(r_2)}_{i,u}[-1])_{(1)}(\displaystyle\sum_{1\leq t\leq l}\limits e_{(t-1)(m+n)+i,(t-r-1)(m+n)+i}[-1])\label{no9}
\end{align}
and
\begin{align}
&(\alpha\sum_{2\leq s\leq l} (s-1)e^{(s)}_{i,i}[-2])_{(1)}\displaystyle\sum_{1\leq t\leq l}\limits e_{(t-1)(m+n)+i+1,(t-r-1)(m+n)+i+1}[-1]\nonumber\\
&\qquad-{(-1)}^{p(e_{i,i+1})}(\alpha\displaystyle\sum_{2\leq s\leq l}\limits (s-1)e^{(s)}_{i,i}[-2])_{(1)}\displaystyle\sum_{1\leq t\leq l}\limits e_{(t-1)(m+n)+i,(t-r-1)(m+n)+i}[-1]).\label{no10}
\end{align}
Let us compute \eqref{no9} and \eqref{no10}. By direct computation, the second term of \eqref{no9} is equal to
\begin{align*}
&-{(-1)}^{p(i+1)}\sum_{\substack{r_1<r_2\\1\leq u\leq m+n\\1\leq t\leq l}}[e^{(r_1)}_{u,i},[e^{(r_2)}_{i,u},e_{(t-1)(m+n)+i,(t-r-1)(m+n)+i}]][-1]\\
&\quad-{(-1)}^{p(i+1)}\sum_{\substack{r_1<r_2\\1\leq u\leq m+n\\1\leq t\leq l}}\kappa(e^{(r_2)}_{i,u},e_{(t-1)(m+n)+i,(t-r-1)(m+n)+i})e^{(r_1)}_{u,i}[-1]\\
&\qquad-{(-1)}^{p(i+1)}\sum_{\substack{r_1<r_2\\1\leq u\leq m+n\\1\leq t\leq l}}{(-1)}^{p(e_{i,u})}\kappa(e^{(r_1)}_{u,i},e_{(t-1)(m+n)+i,(t-r-1)(m+n)+i})e^{(r_2)}_{i,u}[-1]\\
&=-{(-1)}^{p(i+1)}\sum_{1\leq t\leq l}e_{(t-1)(m+n)+i,(t-r-1)(m+n)+i}[-1]+0+0.
\end{align*}
By the similar computation, the first term of \eqref{no9} is equal to 
\begin{equation*}
{(-1)}^{p(i+1)}\sum_{1\leq t\leq l}e_{(t-1)(m+n)+i,(t-r-1)(m+n)+i}[-1].
\end{equation*}
By direct computation, we rewrite the second term of \eqref{no10} as
\begin{align*}
&\quad{(-1)}^{p(e_{i,i+1})}\alpha\displaystyle\sum_{\substack{1\leq s\leq l\\1\leq t\leq l}}\limits (s-1)[e^{(s)}_{i,i}, e_{(t-1)(m+n)+i,(t-r-1)(m+n)+i}][-1]\\
&={(-1)}^{p(e_{i,i+1})}r\alpha\sum_{1\leq t\leq l}e_{(t-1)(m+n)+i,(t-r-1)(m+n)+i}[-1].
\end{align*}
By the similar computation, we find that the first term of \eqref{no10} is zero. Thus, we obtain
\begin{align}
\text{the sum of first two terms of \eqref{no7}}={(-1)}^{p(e_{i,i+1})}r\alpha\sum_{1\leq t\leq l}e_{(t-1)(m+n)+i,(t-r-1)(m+n)+i}[-1].\label{no2021}
\end{align}
Finally, we compute the third term of \eqref{no7}. Since the relation $(\displaystyle\sum_{1\leq s\leq l}\limits(s-1)e^{(s)}_{i,i}[-2])_{(0)}=0$ holds, we can rewrite the third term of \eqref{no7} as
\begin{equation*}
\displaystyle\sum_{1\leq d\leq r}\limits Z_{(1)}(W^{(1)}_{i,i+1})_{(0)}(Z_{(0)})^{r-d}\cdot({(-1)}^{p(i)}\displaystyle\sum_{\substack{r_1<r_2\\1\leq t\leq m+n}}\limits e^{(r_1)}_{t,i}[-1]e^{(r_2)}_{i,t}[-1])_{(0)}(Z_{(0)})^{d-1}W^{(1)}_{i+1,i}.
\end{equation*}
Let us set
\begin{gather*}
T_d=Z_{(1)}(W^{(1)}_{i,i+1})_{(0)}(Z_{(0)})^{r-d},\quad B_d=({(-1)}^{p(i)}\displaystyle\sum_{\substack{r_1<r_2\\1\leq t\leq m+n}}\limits e^{(r_1)}_{t,i}[-1]e^{(r_2)}_{i,t}[-1])_{(0)}(Z_{(0)})^{d-1}W^{(1)}_{i+1,i}
\end{gather*}
Then, the third term of \eqref{no7} is equal to $\displaystyle\sum_{1\leq d\leq r}\limits T_d(B_d)$. 

We rewrite $B_d$ and $T_d$. By \eqref{nom} and \eqref{non}, $T_d$ is the sum of $T^1_d$ and $T^2_d$ such that
\begin{gather}
T_d^1=-\displaystyle\sum_{g=0}^{r-d}\limits \begin{pmatrix} r-d\\g\end{pmatrix}(Z_{(0)})^{r-d-g}(\displaystyle\sum_{1\leq s\leq l}\limits e_{(s+g)(m+n)+i+1,(s-1)(m+n)+i}[-1])_{(1)},\label{noj}\\
T_d^2=(W^{(1)}_{i,i+1})_{(0)}(Z_{(0)})^{r-d}Z_{(1)}.\label{noh}
\end{gather}
Since
\begin{align}
(Z_{(0)})^{d-1}W^{(1)}_{i+1,i}=\sum_{d\leq t\leq l} e_{(t-1)(m+n)+i,(t-d)(m+n)+i+1}[-1].\label{932}
\end{align}
by \eqref{nom} and \eqref{non}, $B_d$ is equal to
\begin{align}
&\quad({(-1)}^{p(i)}\sum_{\substack{1\leq r_1<r_2\leq l\\1\leq u\leq m+n}}e^{(r_1)}_{u,i}[-1]e^{(r_2)}_{i,u}[-1])_{(0)}\sum_{d\leq t\leq l} e_{(t-1)(m+n)+i,(t-d)(m+n)+i+1}[-1]\nonumber\\
&={(-1)}^{p(i)}\sum_{\substack{1\leq r_1<r_2\leq l\\1\leq u\leq m+n}}\sum_{d\leq t\leq l}e^{(r_1)}_{u,i}[-1][e^{(r_2)}_{i,u},e_{(t-1)(m+n)+i,(t-d)(m+n)+i+1}][-1]\nonumber\\
&\quad+\sum_{\substack{1\leq r_1<r_2\leq l\\1\leq u\leq m+n}}\sum_{d\leq t\leq l}{(-1)}^{p(u)}e^{(r_2)}_{i,u}[-1][e^{(r_1)}_{u,i},e_{(t-1)(m+n)+i,(t-d)(m+n)+i+1}][-1]\nonumber\\
&\quad+\sum_{\substack{1\leq r_1<r_2\leq l\\1\leq u\leq m+n}}\sum_{d\leq t\leq l}{(-1)}^{p(u)}\kappa(e^{(r_1)}_{u,i},e_{(t-1)(m+n)+i,(t-d)(m+n)+i+1})e^{(r_2)}_{i,u}[-2].\label{no4}
\end{align}
By direct computation, we find that the first term of the right hand side of \eqref{no4} is equal to
\begin{align}
{(-1)}^{p(i)}\sum_{d\leq r_1<t\leq l} e^{(r_1)}_{i,i}[-1]e_{(t-1)(m+n)+i,(t-d)(m+n)+i+1}[-1]\label{no5}
\end{align}
and the second term of the right hand side of \eqref{no4} is equal to
\begin{align}
&\sum_{\substack{d\leq t<r_2\leq l\\1\leq u\leq m+n}}{(-1)}^{p(u)} e_{i,u}^{(r_2)}[-1]e_{(t-1)(m+n)+u,(t-d)(m+n)+i+1}[-1]\nonumber\\
&\qquad-{(-1)}^{p(i+1)}\sum _{t-d+1<r_2}e_{i,i+1}^{(r_2)}[-1]e_{(t-1)(m+n)+i,(t-d)(m+n)+i}[-1].\label{no6}
\end{align}
By the definition of $\kappa$, the third term of the right hand side of \eqref{no4} is equal to
\begin{align}
\delta_{d,1}\alpha\displaystyle\sum_{1\leq r_2\leq l}(r_2-1)e^{(r_2)}_{i,i+1}[-2].\label{no973}
\end{align}
Adding \eqref{no5}, \eqref{no6}, and \eqref{no973}, we obtain
\begin{align}
B_d&={(-1)}^{p(i)}\sum_{d\leq r_1<t\leq l} e^{(r_1)}_{i,i}[-1]e_{(t-1)(m+n)+i,(t-d)(m+n)+i+1}[-1]\nonumber\\
&\quad+\sum_{\substack{d\leq t<r_2\leq l\\1\leq u\leq m+n}}{(-1)}^{p(u)} e_{i,u}^{(r_2)}[-1]e_{(t-1)(m+n)+u,(t-d)(m+n)+i+1}[-1]\nonumber\\
&\quad-{(-1)}^{p(i)}\sum _{t-d+1<r_2\leq l}e_{i,i+1}^{(r_2)}[-1]e_{(t-1)(m+n)+i,(t-d)(m+n)+i}[-1]\nonumber\\
&\quad+\delta_{d,1}\alpha\displaystyle\sum_{1\leq r_2\leq l}(r_2-1)e^{(r_2)}_{i,i+1}[-2]\nonumber\\
&={(-1)}^{p(i)}\sum_{r_1\neq t} e^{(r_1)}_{i,i}[-1]e_{(t-1)(m+n)+i,(t-d)(m+n)+i+1}[-1])\nonumber\\
&\quad+\sum_{\substack{r_2>t\\u\neq i}}{(-1)}^{p(u)} e_{i,u}^{(r_2)}[-1]e_{(t-1)(m+n)+u,(t-d)(m+n)+i+1}[-1]\nonumber\\
&\quad-{(-1)}^{p(i)}\sum _{t-d+1<r_2\leq l}e_{i,i+1}^{(r_2)}[-1]e_{(t-1)(m+n)+i,(t-d)(m+n)+i}[-1]\nonumber\\
&\quad+\delta_{d,1}\alpha\displaystyle\sum_{1\leq r_2\leq l}(r_2-1)e^{(r_2)}_{i,i+1}[-2].\label{noi}
\end{align}
Now, we compute $T_d(B_d)$. We divide $B_d$ into two parts such that
\begin{align*}
B^1_d&={(-1)}^{p(i)}\sum_{r_1\neq t} e^{(r_1)}_{i,i}[-1]e_{(t-1)(m+n)+i,(t-d)(m+n)+i+1}[-1]\nonumber\\
&\quad+\sum_{\substack{r_2>t\\u\neq i}}{(-1)}^{p(u)} e_{i,u}^{(r_2)}[-1]e_{(t-1)(m+n)+u,(t-d)(m+n)+i+1}[-1]\nonumber\\
&\quad-{(-1)}^{p(i)}\sum _{t-d+1<r_2\leq l}e_{i,i+1}^{(r_2)}[-1]e_{(t-1)(m+n)+i,(t-d)(m+n)+i}[-1],\\
B^2_d&=\delta_{d,1}\alpha\displaystyle\sum_{1\leq r_2\leq l}(r_2-1)e^{(r_2)}_{i,i+1}[-2].
\end{align*}
First, let us compute $T_d(B^2_d)$. By \eqref{nom} and \eqref{non}, we obtain
\begin{align}
T_d(B^2_d)=-\delta_{d,1}{(-1)}^{p(e_{i,i+1})}(r-1)\alpha\sum_{1\leq t\leq l}e_{(t-1)(m+n)+i,(t-r-1)(m+n)+i}[-1].\label{tukare}
\end{align}
Next, let us compute $T_d(B^1_d) =T^1_d(B^1_d)+T^2_d(B^1_d)$. We compute $T^1_d(B_d^1)$ and $T^2_d(B^1_d)$ respectively.
In order to compute $T_d^1(B_d^1)$, we prepare the following three relations;
\begin{align}
&\quad\sum_{1\leq s\leq l}(e_{(s+g)(m+n)+i+1,(s-1)(m+n)+i}[-1])_{(1)}\nonumber\\
&\qquad\qquad\cdot({(-1)}^{p(i)}\sum_{r_1\neq t} e^{(r_1)}_{i,i}[-1]e_{(t-1)(m+n)+i,(t-d)(m+n)+i+1}[-1])\nonumber\\
&=-{(-1)}^{p(i+1)}\sum_{1\leq t\leq l}e_{(t-1)(m+n)+i,(t-g-d-1)(m+n)+i}[-1],\label{no20}\\
&\quad\sum_{1\leq s\leq l}(e_{(s+g)(m+n)+i+1,(s-1)(m+n)+i}[-1])_{(1)}\nonumber\\
&\qquad\qquad\cdot(\sum_{\substack{r_2>t\\u\neq i}} {(-1)}^{p(u)}e_{i,u}^{(r_2)}[-1]e_{(t-1)(m+n)+u,(t-d)(m+n)+i+1}[-1])=0,\label{no21}\\
&\quad\sum_{1\leq s\leq l}(e_{(s+g)(m+n)+i+1,(s-1)(m+n)+i}[-1])_{(1)}\nonumber\\
&\qquad\qquad\cdot({(-1)}^{p(i)}\sum _{t-d+1<r_2\leq l}e_{i,i+1}^{(r_2)}[-1]e_{(t-1)(m+n)+i,(t-d)(m+n)+i}[-1])\nonumber\\
&=-{(-1)}^{p(i+1)}\sum_{1\leq t\leq l}e_{(t-1)(m+n)+i,(t-g-d-1)(m+n)+i}[-1].\label{no22}
\end{align}
We only show the relation \eqref{no22} holds. The other relations are proven similarly.
By direct computation, \eqref{no22} is equal to
\begin{align*}
&\quad{(-1)}^{p(i)}\displaystyle\sum_{1\leq s\leq l}\limits\displaystyle\sum _{t-d+1<r_2\leq l}\limits[[e_{(s+g)(m+n)+i+1,(s-1)(m+n)+i},e_{i,i+1}^{(r_2)}],e_{(t-1)n+i,(t-d)n+i})][-1]\\
&=-{(-1)}^{p(i+1)}\sum_{1\leq t\leq l}e_{(t-1)(m+n)+i,(t-g-d-1)(m+n)+i}[-1].
\end{align*}
Thus, we have obtained \eqref{no22}. By \eqref{no20}-\eqref{no22} and \eqref{noj}, we find the relation
\begin{equation}
T_d^1(B^1_d)=0.\label{nikui}
\end{equation}
Similarly to \eqref{no20}-\eqref{no22}, we obtain the following three equations;
\begin{align}
&\sum_{1\leq s\leq l}(e_{s(m+n)+i,(s-1)(m+n)+i}[-1])_{(1)}\nonumber\\
&\qquad\qquad\cdot({(-1)}^{p(i)}\sum_{r_1\neq t} e^{(r_1)}_{i,i}[-1]e_{(t-1)(m+n)+i,(t-d)(m+n)+i+1}[-1])\nonumber\\
&=-{(-1)}^{p(i)}\sum_{1\leq t\leq l}e_{(t-1)(m+n)+i,(t-d-1)(m+n)+i+1}[-1],\label{no23}\\
&\displaystyle\sum_{1\leq s\leq l}(e_{s(m+n)+i,(s-1)(m+n)+i}[-1])_{(1)}\nonumber\\
&\qquad\qquad\cdot(\sum_{\substack{r_2>t\\u\neq i}} {(-1)}^{p(u)}e_{i,u}^{(r_2)}[-1]e_{(t-1)(m+n)+u,(t-d)(m+n)+i+1}[-1])=0,\label{no24}\\
&\displaystyle\sum_{1\leq s\leq l}(e_{s(m+n)+i,(s-1)(m+n)+i}[-1])_{(1)}\nonumber\\
&\qquad\qquad\cdot({(-1)}^{p(i)}\sum _{t-d+1<r_2\leq l}e_{i,i+1}^{(r_2)}[-1]e_{(t-1)(m+n)+i,(t-d)(m+n)+i}[-1])=0.\label{no25}
\end{align}
By \eqref{no23}-\eqref{no25} and \eqref{noh}, we obtain
\begin{align}
T_d^2(B^1_d)&=-{(-1)}^{p(i)}(W^{(1)}_{i,i+1})_{(0)}(Z_{(0)})^{r-d}\sum_{1\leq t\leq l}e_{(t-1)(m+n)+i,(t-d-1)(m+n)+i+1}[-1]\nonumber\\
&=-{(-1)}^{p(i)}\sum_{1\leq t\leq l}e_{(t-1)(m+n)+i+1,(t-r-1)(m+n)+i+1}[-1]\nonumber\\
&\quad+{(-1)}^{p(i+1)}\sum_{1\leq t\leq l}e_{(t-1)(m+n)+i,(t-r-1)(m+n)+i}[-1],\label{nami}
\end{align}
where the second equality is due to \eqref{nom} and \eqref{non}. By \eqref{tukare}, \eqref{nikui} and \eqref{nami}, we have
\begin{align}
\displaystyle\sum_{1\leq d\leq r}\limits T_d(B_d)&=-{(-1)}^{p(e_{i,i+1})}(r-1)\alpha\sum_{1\leq t\leq l}e_{(t-1)(m+n)+i,(t-r-1)(m+n)+i}[-1]\nonumber\\
&\quad-{(-1)}^{p(i)}r\sum_{1\leq t\leq l}e_{(t-1)(m+n)+i+1,(t-r-1)(m+n)+i+1}[-1]\nonumber\\
&\quad+{(-1)}^{p(i+1)}r\sum_{1\leq t\leq l}e_{(t-1)(m+n)+i,(t-r-1)(m+n)+i}[-1].\label{2022}
\end{align}
Adding \eqref{no2021} and \eqref{2022}, \eqref{no7} is equal to 
\begin{align*}
&{(-1)}^{p(e_{i,i+1})}\alpha\displaystyle\sum_{1\leq t\leq l}\limits e_{(t-1)(m+n)+i,(t-r-1)(m+n)+i}[-1]\\
&\quad-{(-1)}^{p(i+1)}r\sum_{1\leq t\leq l}e_{(t-1)(m+n)+i,(t-r-1)(m+n)+i}[-1]\nonumber\\
&\quad+{(-1)}^{p(i)}r\displaystyle\sum_{1\leq t\leq l}\limits e_{(t-1)(m+n)+i+1,(t-r-1)(m+n)+i+1}[-1]+\text{higher terms}.
\end{align*}
We have obtained \eqref{929}.
\end{proof}
Since we complete the proof of Claims~\ref{T1} and \ref{T3}, we have proved Theorem~\ref{Tinf}.

\end{appendices}

\section*{Acknowledgement}
The author wishes to express his gratitude to his supervisor Tomoyuki Arakawa for suggesting lots of advice to improve this paper. We express our sincere thanks to Ryosuke Kodera for carefully reading the manuscript and giving me lots of advice and comments. The author is particularly grateful for the assistance given by Naoki Genra and Shigenori Nakatsuka. This work was supported by Iwadare Scholarship and and JSPS KAKENHI, Grant-in-Aid for JSPS Fellows, Grant Number JP20J12072. 
\section*{Data Availability}
The authors confirm that the data supporting the findings of this study are available within the article and its supplementary materials.
\bibliographystyle{plain}
\bibliography{syuu}
\end{document}